\documentclass[11pt,a4paper]{article}
\usepackage[a4paper,top=3cm,bottom=4.5cm,left=2.5cm,right=2.5cm]{geometry}
\usepackage{graphicx,caption}
\usepackage{enumerate}
\usepackage{amssymb}

\def\XXint#1#2#3{{\setbox0=\hbox{$#1{#2#3}{%
				\int}$ }
		\vcenter{\hbox{$#2#3$ }}\kern-.6\wd0}}

\usepackage{mathrsfs}
\usepackage{amsmath}
\usepackage{latexsym}
\usepackage{amsthm}
\usepackage{eucal}
\usepackage{eufrak}
\usepackage{relsize}
\newtheorem{definition}{Definition}[section]

\newtheorem{lemma}[definition]{Lemma}

\newtheorem{teo}[definition]{Theorem}

\newtheorem{prop}[definition]{Proposition}

\newtheorem{corollary}[definition]{Corollary}
\theoremstyle{remark}
\newtheorem{remark}[definition]{Remark}
\newtheorem{example}[definition]{Example}

\numberwithin{equation}{section}
\DeclareMathOperator{\expo}{expo}
\DeclareMathOperator{\extr}{extr}
\DeclareMathOperator{\epi}{epi}

\title{A new model for suspension bridges\\
	involving the convexification of the cables}
\author{Graziano CRASTA$^{\dagger}$ -- Alessio FALOCCHI$^\ddagger$ -- Filippo GAZZOLA$^\ddagger$\\
	{\small $^\dagger$Dipartimento di Matematica "G. Castelnuovo" - Sapienza Universit\`a di Roma,}\vspace{-3mm}\\
	{\small P.le Aldo Moro 5 - 00185 Roma, Italy}\\
	{\small $^\ddagger$Dipartimento di Matematica - Politecnico di Milano,}\vspace{-3mm}\\
	{\small Piazza Leonardo da Vinci 32 - 20133 Milano, Italy}\\
	{\small {\tt graziano.crasta@uniroma1.it - alessio.falocchi@polimi.it - filippo.gazzola@polimi.it}}}
\date{}

\begin{document}

\maketitle

\begin{abstract}
The final purpose of this paper is to show that, by inserting a convexity constraint on the cables of a suspension bridge, the torsional instability of the deck
appears at lower energy thresholds. Since this constraint is suggested by the behavior of real cables, this model appears more reliable than the
classical ones. Moreover, it has the advantage to reduce to two the number of degrees of freedom (DOF), avoiding to introduce the slackening mechanism
of the hangers. The drawback is that the resulting energy functional is extremely complicated, involving the convexification of unknown functions.
This paper is divided in two main parts. The first part is devoted to the study of these functionals, through classical methods of calculus of variations.
The second part applies this study to the suspension bridge model with convexified cables.\par \textbf{Keywords:} suspension bridges, instability, convexification.\par
\textbf{AMS Subject Classification (MSC2010)}: 35C31, 74B20.
\end{abstract}

\section{Introduction}
A suspension bridge is composed by four towers, a rectangular deck, two sustaining cables and a large number of hangers, see Figure \ref{ponte}
for a sketch of the side view. In the reference system $(O,x,y)$ the vertical displacement $w$ is positive downwards while $x$ is oriented
horizontally along the deck.
\begin{figure}[htbp]
	\centering
	\includegraphics[width=8cm]{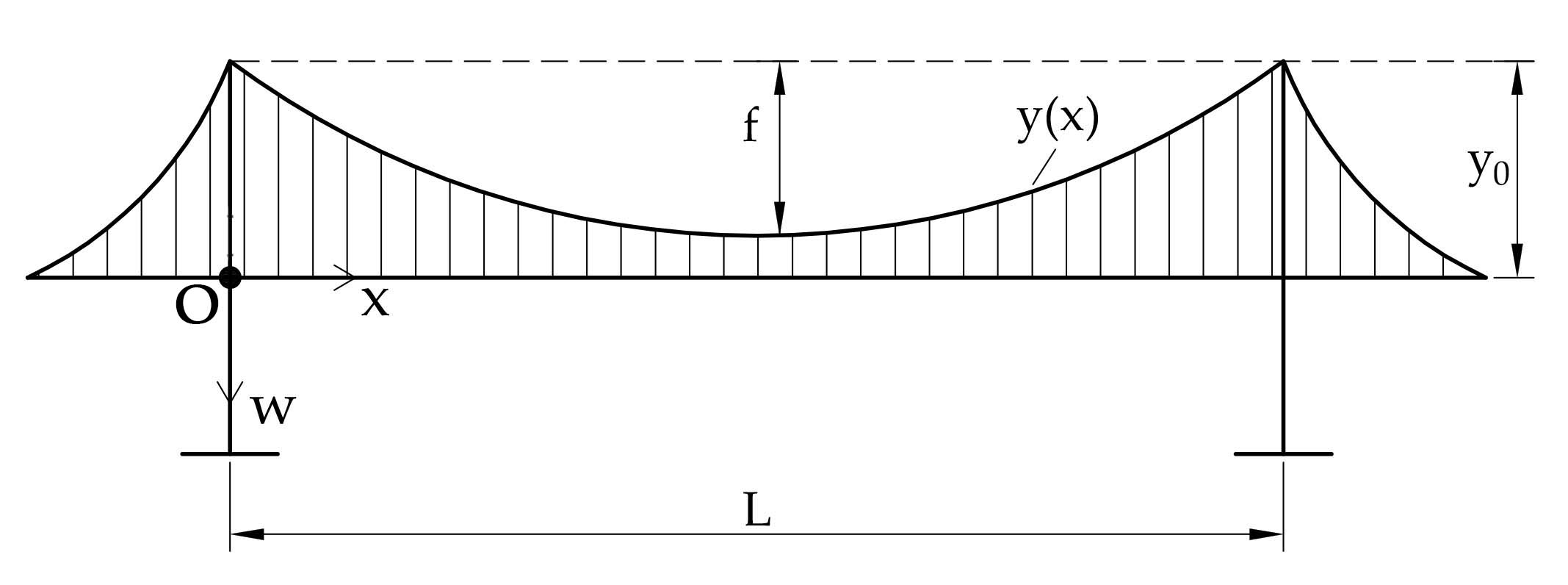}
	\caption{Sketch of the side view of a suspension bridge.}
	\label{ponte}
\end{figure}
We view the part of the deck between the towers as a degenerate plate occupying at rest the planar position $(0,L)\times(-\ell,\ell)$,
and composed by a central beam of length $L$ and by cross sections of length $2\ell\ll L$ whose midpoints lie on the beam. Each cross section
is free to rotate around the beam and to leave the horizontal position. The hangers link the endpoints of the cross sections (the long edges of
the plate) to the cables. This model is called {\em fish-bone} in \cite{berchio} and a linear version of it was suggested in
\cite[p.458, Chapter VI]{yakubovich}.\par
While the mathematical community \cite{berkovits,drabek1,drabek2,glover,holubova17,holubova,lazer,mckenna} is prone to take into account the
hangers slackening, the elastic deformation of the hangers is usually neglected in the engineering literature, this simplification being only partially
justified by precise studies on linearized models. The hangers are considered as rigid bars so that the deck and the cables undergo
the same movement. Nevertheless, this assumption is unreasonable since the hangers resist to traction but not to compression.
Slackening of the hangers was observed by Farquharson \cite[V-12]{TNB} during the Tacoma Narrows Bridge (TNB) collapse.
In the model with rigid hangers considered in \cite{falocchi2},
\begin{figure}[htbp]
	\centering
	\includegraphics[width=13cm]{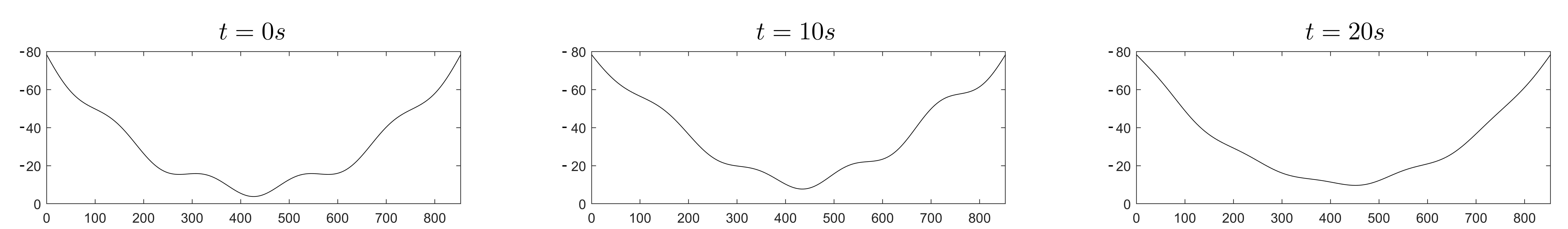}
	\caption{In case of rigid hangers, cable shape with the deck oscillating on the 9th longitudinal mode of initial amplitude $3.87$m,
		see \cite{falocchi2}.}
	\label{cavo shape}
\end{figure}
the cables displayed shapes similar to those depicted in Figure \ref{cavo shape},
that reproduce the shape every 10s on $[0,L]$ (with $L=853.44$m as for the TNB). The nonconvex shape becomes more evident as the energy in the system (the amplitude of oscillation) increases.
For the plots in Figure \ref{cavo shape}, the oscillating mode of the deck is the 9th and the initial amplitude is $3.87$m, which lies in a physical range.
Whence, the assumption of rigid hangers leads to unrealistic pictures never seen in real bridges. The cables are always convex due
to their mass and to the hangers slackening. A nonconvex configuration would also increase the tension of the cable, against the principle
of minimal energy. For these reasons, as in \cite{sperone}, we will assume that the actual shape of
the cables coincides with its \textit{convexified form}, namely the shape minimizing the length under the same load conditions. Acting only on
this geometric feature, we propose a two DOF model in which the slackening of the hangers is considered indirectly. The great advantage is that
we do not need an explicit nonlinearity describing the slackening mechanisms.\par
The drawback is that the convexity constraint leads to some technical mathematical difficulties, see \cite{bucur,buttazzo,cellina2}
and also \cite{alves,colasuonno} for different but related equations. This is why Section
\ref{s:convex} is devoted to general results related to the convexification of one dimensional
functions. In Theorem \ref{prop variaz gen} we compute the variation of functionals containing a convexification. This characterization is new and, in
our opinion, of independent interest with possible applications to more general variational problems. The convexification makes the energy
function non-differentiable: its variation yields a weak form of a \textit{system of partial differential inclusions} (see (\ref{eq weak sist2}) in Section
\ref{The system of partial differential inclusions}) for which the uniqueness of the solution is not expected. However, by exploiting the
peculiarity of the model, we are able to show that Galerkin approximation of the problem admits a unique classical solution, because
the obstruction to the differentiability of the energy is ruled out in a finite dimensional phase space. This suggests to introduce the class of
\textit{approximable solutions} of the problem, namely solutions that are the limit of the Galerkin subsequences, see Definition \ref{defapproximable} in
Section \ref{The system of partial differential inclusions}. This class of solutions will be physically justified and it will be shown that they
are representative of the full problem; moreover, within this class we are able to obtain existence results, see Theorem \ref{corol inclusione}.
This requires some particular attention due to the convexification and to the unusual behavior of test functions.\par
The torsional oscillations of the deck were the main cause for the TNB collapse \cite{TNB} and of several other collapses, see e.g.\ \cite{bookgaz}.
A new mathematical explanation for the origin of torsional
oscillations was given in \cite{argaz} through the introduction of suitable Poincar\'e maps: these oscillations appear whenever there is a large amount
of energy within the bridge and this happens due to the nonlinear behavior of structures. The model in \cite{argaz} was fairly simplified, but the very
same conclusion was subsequently reached in more sophisticated models \cite{Gazzola hyperb,Gazzola Torsion,berchio,capsoni,falocchi,falocchi2}.
A further purpose of the paper is to study the torsional instability of the deck through the model with convexified cables. To this end, we proceed numerically by
introducing a new algorithm dealing with the convexification at the beginning of each temporal iteration. We then numerically show that the slackening
mechanism hidden in the convexification of the cables yields energy thresholds of instability for high modes significantly smaller than in models where slackening is
neglected. This means that \textit{the slackening of the hangers must be taken into account because it gives lower thresholds of torsional instability}.\par
This paper is organized as follows. In Section \ref{s:convex} we
recall some features of the convexification of a function: we merely focus on the situation that applies to cables since the general setting
is fairly complicated. In Section \ref{Energy involved in the structure} we complete the analysis of the model
through a careful energy balance. This enables us to derive the differential inclusions and the differential equations related to approximable solutions in Section \ref{The variational problem}.
In Section \ref{Numerical analysis} we quote our numerical experiments and results. Sections \ref{proof prop3} and \ref{Proof of the Theorem 2} are devoted to the proofs of our results. Finally, in Section \ref{concl} we outline the conclusions.
Throughout this paper we denote the derivatives of a function $f=f(t)$ (depending only on $t$), of a function $g=g(x)$ (depending only on $x$)
and the partial derivatives of a function $w=w(x,t)$, respectively by
\begin{equation*}
\dot{f}=\dfrac{df}{dt}, \hspace{6mm} g'=\dfrac{dg}{dx}, \hspace{6mm}
w_{x}=\frac{\partial w}{\partial x}, \hspace{4mm} w_{t} =\frac{\partial w}{\partial t},
\end{equation*}
and similarly for higher order derivatives.

\section{The one-dimensional variational problem}
\label{s:convex}


Let $\mathcal{I}=(a,b)\subset\mathbb{R}$ be an open bounded interval. Since we are interested in the specific application of a \textit{real physical} cable, whose shape is described by a function in $H^2(\mathcal{I})\subset C^1(\overline{\mathcal{I}})$,
in all the paper we shall consider only profiles $f$ of class
\( C^1(\overline{\mathcal{I}})\),
avoiding more general assumptions on $f$.
%

We shall denote by $f^{**}$ the convex envelope (or convexification)
of $f$, i.e.\ the largest convex function satisfying $f^{**}\leq f$
in $\overline{\mathcal{I}}$.
Since $f$ is of class $C^1$, then also $f^{**}$ belongs to
$C^1(\overline{\mathcal{I}})$.

In the sequel a major role will be played by the maximal intervals
where $f^{**}$ is affine.
We denote by $K^i = [c^i, d^i]$, $i \in J_C$, the (possibly countable)
family of all these intervals.
Let $K_f\subset\mathcal{I}$ be the contact set of $f$, i.e.
\[
K_f := \{x\in\mathcal{I}:\
f(x) = f^{**}(x)\}\
\]
and note that $c^i, d^i \in K_f \cup \{a, b\}$.
The set $N := \mathcal{I}\setminus K_f$ is the union
of an at most countable family $I_i := (a^i, b^i)$, $i\in J$ of
open intervals.

We also use the notation
\[
\widetilde{K}_f := K_f \setminus\bigcup_{i\in J_C} K^i.
\]
Around points $x\in\widetilde{K}_f$ the function $f$ is strictly convex,
meaning that
\[
f^{**}(x) > f(x_0) + f'(x_0)\, (x-x_0),
\qquad
\forall x\in [a,b],\ x\neq x_0.
\]
More precisely, the set
$\{(x, f^{**}(x)):\ x \in \widetilde{K}_f\cup\{a,b\}\}$
coincides with the set
of exposed points of the epigraph of $f^{**}$, see
Section~\ref{proof prop3} for the precise definitions.
\subsection{The variation of functionals of convexified functions}
\label{variation of convexified function}

In order to study the behavior of the cables, we need to compute the variation of energies depending on the convexification of a function.
We deal with functionals such as $u\mapsto\int_{\mathcal{I}}[\Lambda(u)]^{**}dx$ with $\Lambda\in C^1(\mathbb{R})$ and
we need to compute the Gateaux derivative of such functionals. As we shall see,
in general these functionals are not Gateaux differentiable at every point.
To illustrate this phenomenon, let us consider first
the particular case $\Lambda(u) = u$.

\begin{prop}
	\label{p:prop variaz}
	Let $f\in C^1(\overline{\mathcal{I}})$ and let $f^{**}$, 
	$K^i = [c^i, d^i]$ ($i\in J_C$) and $\widetilde{K}_f$
	be as above.
	Let $\varphi\in C^{\infty}_c(\mathcal{I})$,
	and, for $i\in J_C$, consider the extended real-valued functions
	\begin{equation*}\label{f:phi}
	\varphi_i^\pm\colon K^i \to \overline{\mathbb{R}},
	\qquad
	\varphi_i^\pm(x) :=
	\begin{cases}
	\varphi(x) & x \in K^i \cap (K_f \cup \{a,b\}),\\
	\pm\infty & x \in K^i \setminus (K_f \cup \{a,b\}).
	\end{cases}
	\end{equation*}
	Then we have
	\begin{equation}
	\label{gateaux-lin}
	\lim\limits_{s\rightarrow 0^\pm}
	\int_{\mathcal{I}}\frac{(f+s\varphi)^{**}-f^{**}}{s}\hspace{1mm}dx
	= \int_{\mathcal{I}}\mathcal{J}^{\varphi}_\pm\hspace{1mm}dx,
	\end{equation}
	where
	\begin{equation}
	\label{zbar2}
	\mathcal{J}^\varphi_\pm(x) :=
	\begin{cases}
	\pm(\pm\varphi_i^\pm)^{**}(x) & x \in K^i,\ i\in J_C, \\
	\varphi(x) & x \in \widetilde{K}_f.
	\end{cases}
	\end{equation}
\end{prop}

Proposition \ref{p:prop variaz}, whose proof is given in Section \ref{proof prop3}, has an important consequence.

\begin{corollary}
	\label{prop variaz}
	Under the same assumptions of Proposition~\ref{p:prop variaz},
	the functional $f \mapsto \int_{\mathcal{I}} f^{**}$
	is Gateaux--differentiable at $f$ if and only if
	\begin{equation}\label{noflat}
	\overline{K_f} = \overline{\widetilde{K}_f},\quad
	\text{i.e., $f > f^{**}$ on any open interval where $f^{**}$ is affine}.
	\end{equation}
	In this case,
	for every $\varphi\in C^{\infty}_c(\mathcal{I})$
	it holds that
	$\mathcal{J}^\varphi_+ = \mathcal{J}^\varphi_- =: \mathcal{J}^\varphi$,
	with
	\begin{equation}
	\label{zbar}
	\mathcal{J}^\varphi(x) :=
	\begin{cases}
	\varphi(a^i)+\dfrac{\varphi(b^i)-\varphi(a^i)}{b^i-a^i}(x-a^i) & x\in I^i,
	\ i \in J,\\
	\varphi(x) & x \in {K}_f.
	\end{cases}
	\end{equation}
\end{corollary}

\begin{remark}
	When condition (\ref{noflat}) is satisfied,
	the intervals $I^i$
	coincide with the interior of the intervals $K^i$, i.e. one has
	that $a^i = c^i$ and $b^i = d^i$ for every $i\in J$.
	Also note that if $f\in C^1(\overline{\mathcal{I}})$ and $\varphi\in C^{\infty}_c(\mathcal{I})$, then $\mathcal{J}^\varphi_\pm, \mathcal{J}^\varphi\in W^{1,1}(0,L)$.	
	\label{rmk 2.5}
\end{remark}

If $f$ is convex the shape of $\varphi$ is maintained and we have the classical Gateaux derivative. Proposition \ref{p:prop variaz} states that the shape
of the test function $\varphi$ may change if the variation involves a convexification, see Figure \ref{zf}b). This possible change of $\varphi$ makes the
problem very challenging and is the price for having a physically reliable modeling of the cables.
The next example explains why assumption \eqref{noflat} is necessary in order to have the Gateaux--differentiability.

\begin{example}
	For some $\mu, \upsilon\in\mathbb{R}$, take $f(x)=\mu x+\upsilon$ on 
	$\mathcal{I}=(-2,2)$ and let
	\begin{equation*}
	\varphi(x)=	e^{\frac{1}{x^2-1}}\mbox{ if }x\in (-1,1),\qquad\varphi(x)=0\mbox{ if }x\in\overline{\mathcal{I}}\setminus (-1,1),
	\qquad\qquad  \big(\varphi\in C_c^\infty(\mathcal{I})\big).
	\end{equation*}
	The limits (\ref{gateaux-lin}) depend on the sign of $s$. Indeed,
	\begin{equation}
	\label{gateaux2}
	\lim\limits_{s\rightarrow 0\pm}\int_{\mathcal{I}}\frac{(s\varphi)^{**}}{s}\hspace{1mm}dx
	= \int_{\mathcal{I}}\pm(\pm\varphi)^{**}dx
	\end{equation}
	and if $s>0$ we have $(s\varphi)^{**}\equiv 0$ so that (\ref{gateaux2}) vanishes,
	while if $s<0$ we have that $(s\varphi)^{**} = s[-(-\varphi)^{**}]$ and we obtain the point $\zeta\approx 0.25$, such that
	\begin{equation*}
	-(-\varphi)^{**}(x)=\varphi(x)\mbox{ if}|x|\in[0,\zeta],\qquad-(-\varphi)^{**}(x)=\dfrac{e^{\frac{1}{\zeta^2-1}}}{\zeta-2}(|x|-2)\mbox{ if }|x|\in(\zeta,2),
	\end{equation*}
	see Figure \ref{esempio}. It is readily seen that the right and left limits of \eqref{gateaux2} are different,
	implying the non-existence of the Gateaux derivative.
	\hfill $\square$
\begin{figure}[h]
	\centering
	\includegraphics[width=125mm]{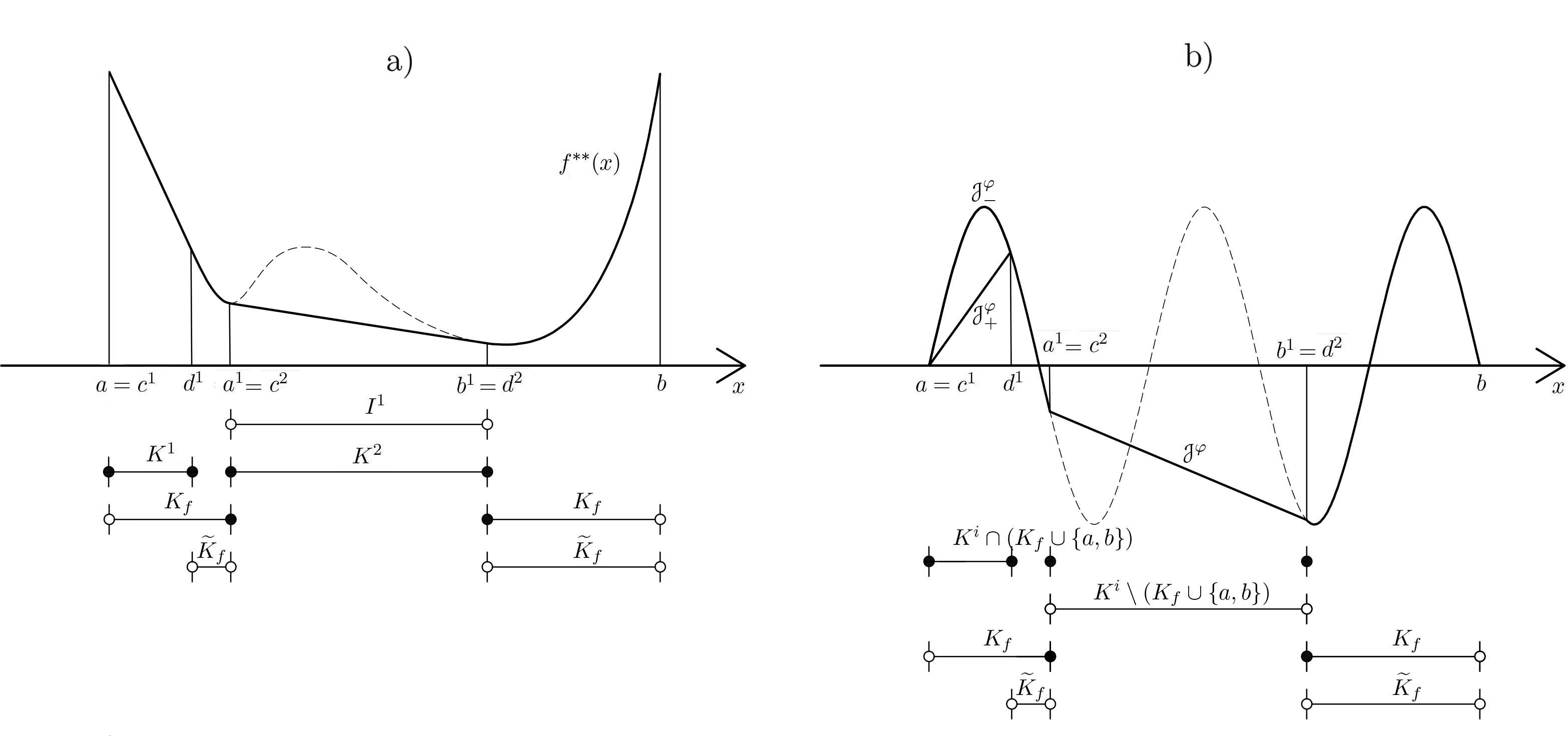}
	\caption{a) Convexification of a function $f$. b) The corresponding $\mathcal{J}^\varphi_\pm(x)$, $\mathcal{J}^\varphi(x)$ as in (\ref{zbar2}), (\ref{zbar}).}
	\label{zf}
\end{figure}
	\begin{figure}[htbp]
		\centering
		\includegraphics[width=8cm]{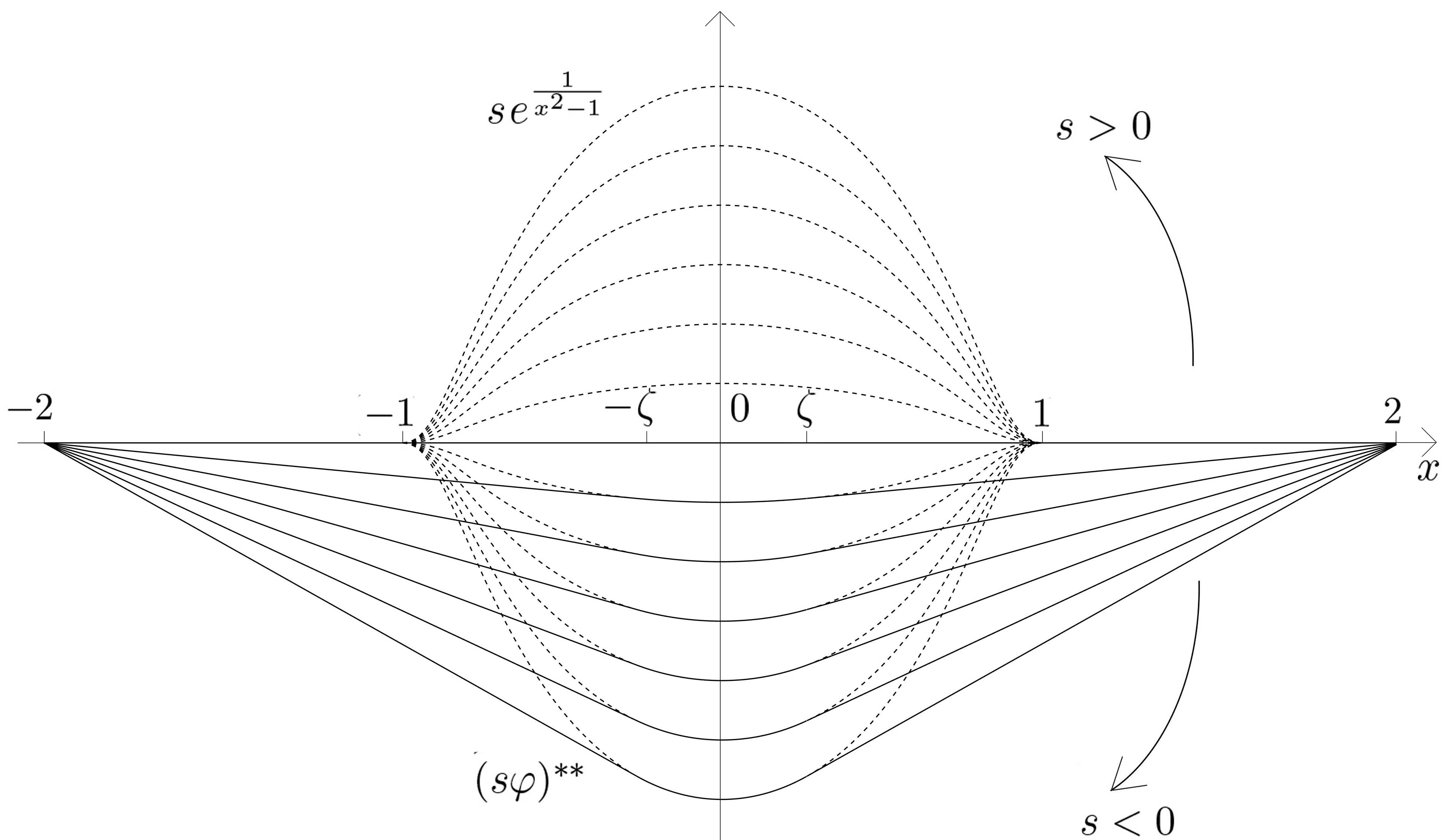}
		\caption{Plot of $(s\varphi)^{**}$ and $s\varphi$ (dashed), for some values of the parameter $s$.}
		\label{esempio}
	\end{figure}
\end{example}

The main result of this section is the following generalization of Proposition \ref{p:prop variaz}:

\begin{teo}
	Consider $u\in C^1(\overline{\mathcal{I}},\mathbb{R})$, $\Lambda\in C^1(\mathbb{R})$ and let $\Lambda'$ be its derivative.
	Let $f:=\Lambda\circ u:\overline{\mathcal{I}}\rightarrow \mathbb{R}$ and let $f^{**}$, $I^i$ ($i\in J$), $K^i$ ($i\in J_C$) and $\widetilde{K}_f$
	be as above.
	Furthermore, assume that $f$ satisfies \eqref{noflat}. Then, for all
	$\varphi\in C^{\infty}_c(\mathcal{I})$, we have
	\begin{equation*}\label{zbar gen}
	\lim\limits_{s\rightarrow0}
	\int_{\mathcal{I}}\frac{[\Lambda(u+s\varphi)]^{**}-[\Lambda(u)]^{**}}{s}\hspace{1mm}dx=\int_{\mathcal{I}}\mathcal{G}^{u,\varphi}\hspace{1mm}dx,
	\end{equation*}
	where \small
	\begin{equation*}
	\mathcal{G}^{u,\varphi}(x):=
	\begin{cases}
	\varphi(a^i)\Lambda'\big(u(a^i)\big)+\dfrac{\varphi(b^i)\Lambda'\big(u(b^i)\big)-\varphi(a^i)\Lambda'\big(u(a^i)\big)}{b^i-a^i}(x-a^i) &
	x\in I^i,\\
	\varphi(x)\Lambda'\big(u(x)\big) & x\in\overline{\mathcal{I}}\setminus \bigcup\limits_{i\in J} I^i.
	\end{cases}
	\label{Gcx}
	\end{equation*}
	\label{prop variaz gen}
\end{teo}
\normalsize
Theorem \ref{prop variaz gen}, whose proof is given in Section \ref{proof prop3}, has an instructive application.
\begin{example}
	If $\theta\in C^1(\overline{\mathcal{I}})$, $\Lambda(\theta)=\sin\theta$, $\psi\in C^\infty_c(\mathcal{I})$, Theorem \ref{prop variaz gen} yields
	\begin{equation*}
	\lim\limits_{s\rightarrow 0}\int_{\mathcal{I}}\frac{[\sin(\theta+s\psi)]^{**}-[\sin\theta]^{**}}{s}\hspace{1mm}dx=\int_{\mathcal{I}}\mathcal{G}^{\theta,\psi}\hspace{1mm}dx,
	\label{sintheta}
	\end{equation*}
	with
	\small
	\begin{equation}
	\mathcal{G}^{\theta,\psi}(x):=
	\begin{cases}
	\psi(a^i)\cos\big(\theta(a^i)\big)+\dfrac{\psi(b^i)\cos\big(\theta(b^i)\big)-\psi(a^i)\cos\big(\theta(a^i)\big)}{b^i-a^i}(x-a^i) & x\in I^i,\\\
	\psi(x)\cos\big(\theta(x)\big) & x\in\overline{\mathcal{I}}\setminus \bigcup\limits_{i\in J} I^i.
	\end{cases}
	\label{Gthetaz}
	\end{equation}
	\label{es theta}
\end{example}
\normalsize

\subsection{Properties of the projection on the cone of convex functions}
\label{Properties of the projection on the cone of convex functions}
\def\Fss{F^{**}}
\def\Gss{G^{**}}
In this section we give some properties of convexified functions that we will use in the sequel to obtain a priori estimates. In the sequel we denote by $\|\cdot\|_p$ the norm related to the Lebesgue space  $L^p(a,b)$ with $1\leq p\leq \infty$. All the proofs are given in Section \ref{proof prop3}.
\begin{prop}\label{l:base}
	Let $T\colon C^0([a,b]) \to C^0([a,b])$ be the operator defined by
	\begin{equation}\label{f:defT}
	Tf := (\Fss)',	\qquad	\text{where}\ F(x):= \int_a^x f(y)\, dy,	\quad y\in [a,b].
	\end{equation}
	Then
	\begin{eqnarray}
	\int_a^b |Tf - Tg| \leq	\int_a^b |f-g|	\qquad \forall f,g\in C^0([a,b]),\label{f:th} \\
	\|Tf - Tg\|_1	\leq \|f-g\|_1	\qquad	\forall f,g\in L^1(a,b). 	\label{f:th2}
	\end{eqnarray}
\end{prop}

Proposition \ref{l:base} shows that the map $T$ is Lipschitzian from $L^1$ to $L^1$ and it enables us to prove that the convexification is Lipschitzian
from $W^{1,1}_0$ to $W^{1,1}_0$.

\begin{corollary}
	The operator $P\colon W^{1,1}_0(a,b) \to W^{1,1}_0(a,b)$, defined by
	$P[F] := \Fss$, is Lipschitz continuous.
	More precisely,
	\[
	\|\Fss - \Gss\|_{W^{1,1}}
	\leq \left(\frac{b-a}{2} + 1\right) \|F' - G'\|_1
	\qquad
	\forall F,G\in W^{1,1}_0(a,b).
	\]
\end{corollary}

In the sequel we denote by $\mathcal{J}^\varphi_F$, $\mathcal{G}^{F,\psi}$ and $\mathcal{J}^\varphi_G$, $\mathcal{G}^{G,\psi}$ the corresponding functions associated respectively to $F$ and $G$ as in \eqref{zbar} and \eqref{Gthetaz}. About the regularity of $\mathcal{J}^\varphi_F$ and $\mathcal{J}^\varphi_G$ we refer to Remark \ref{rmk 2.5} and similarly for $\mathcal{G}^{F,\psi}$ and $\mathcal{G}^{G,\psi}$.
The next statement is crucial for the existence and uniqueness result in Section \ref{The system of partial differential inclusions}.
\begin{prop}
	Let $T\colon L^1(a,b)\to L^1(a,b)$ be as in \eqref{f:defT}.
	Then\small
	\begin{equation*}
	\label{prop Iphi}
	\bigg|\int_a^b [Tf\hspace{1mm} (\mathcal{J}^\varphi_F)' - Tg\hspace{1mm} (\mathcal{J}^\varphi_G)']dx\bigg|
	\leq \|\varphi'\|_\infty\|f-g\|_1
	\qquad
	\forall f,g\in L^1(a,b),\ \forall \varphi\in C^{\infty}_c(\mathcal{I}).
	\end{equation*}
	\label{lipschitz Iphi}
\end{prop}
\normalsize
Similarly, it is possible to state the following more general result.
\begin{prop}
	Let  $\Lambda$ and $\mathcal{G}^{\theta,\psi}$ be as in Example \ref{es theta},
	$\mathcal{H}\in {\rm Lip}(\mathbb{R})$ with Lipschitz constant $\mathcal{L}>0$. Then:\footnotesize
	\begin{enumerate}[i)]
		\item	\hspace{3mm}$\bigg|\int_a^b [\mathcal{H}(Tf)\hspace{1mm} (\mathcal{J}^\varphi_F)' - \mathcal{H}(Tg)\hspace{1mm} (\mathcal{J}^\varphi_G)']dx\bigg|
		\leq \mathcal{L}\|\varphi'\|_\infty\|f-g\|_1
		\quad
		\forall f,g\in L^1(a,b),\ \forall \varphi\in C^{\infty}_c(\mathcal{I});$
		\item $\exists C>0,\ \bigg|\int_a^b [\mathcal{H}(Tf)\hspace{1mm} (\mathcal{G}^{F,\psi})' - \mathcal{H}(Tg)\hspace{1mm} (\mathcal{G}^{G,\psi})']dx\bigg|
		\leq C\|F-G\|_{W^{1,1}}\quad	\forall f,g\in L^1,\ \forall \psi\in C^{\infty}_c(\mathcal{I})$.
	\end{enumerate}
	\normalsize
	\label{rmk}
\end{prop}

We conclude this section with the continuous dependence of $(\mathcal{J}^\varphi)'$ on $f$.

\begin{prop}\label{Iphi conv}
	Let $f, f_n\in C^1(\overline{\mathcal{I}})$, $n\in\mathbb{N}$,
	satisfy assumption \eqref{noflat},
	assume that the sequence $\{f_n\}$ converges uniformly to $f$,
	and let $\varphi\in C^{\infty}_c(\mathcal{I})$.
	Denote by $\mathcal{J}^\varphi$ the function related to $f$
	defined in \eqref{zbar}
	and by $\mathcal{J}^\varphi_n$  the corresponding function related to $f_n$.
	Then $\|\mathcal{J}^\varphi_n-\mathcal{J}^\varphi\|_1 \to 0$ and $\|(\mathcal{J}^{\varphi}_n)'-(\mathcal{J}^\varphi)'\|_1 \to 0$.
\end{prop}

\section{Energy balance in a suspension bridge}
\label{Energy involved in the structure}

\subsection{The energy of the deck}
\label{Kinetic, potential and deformation energy of the system}
In this section we define all the energetic contributions involved in the cable-hangers-beam system aiming to derive the variational form of the problem.
In Figure \ref{schema} is sketched a cross section of the bridge, in which the degrees of freedom $w(x,t)$ and $\theta(x,t)$ correspond
respectively to the downward displacement and the torsional angle around the barycentric line of the deck.
\begin{figure}[htbp]
	\centering
	\includegraphics[width=13cm]{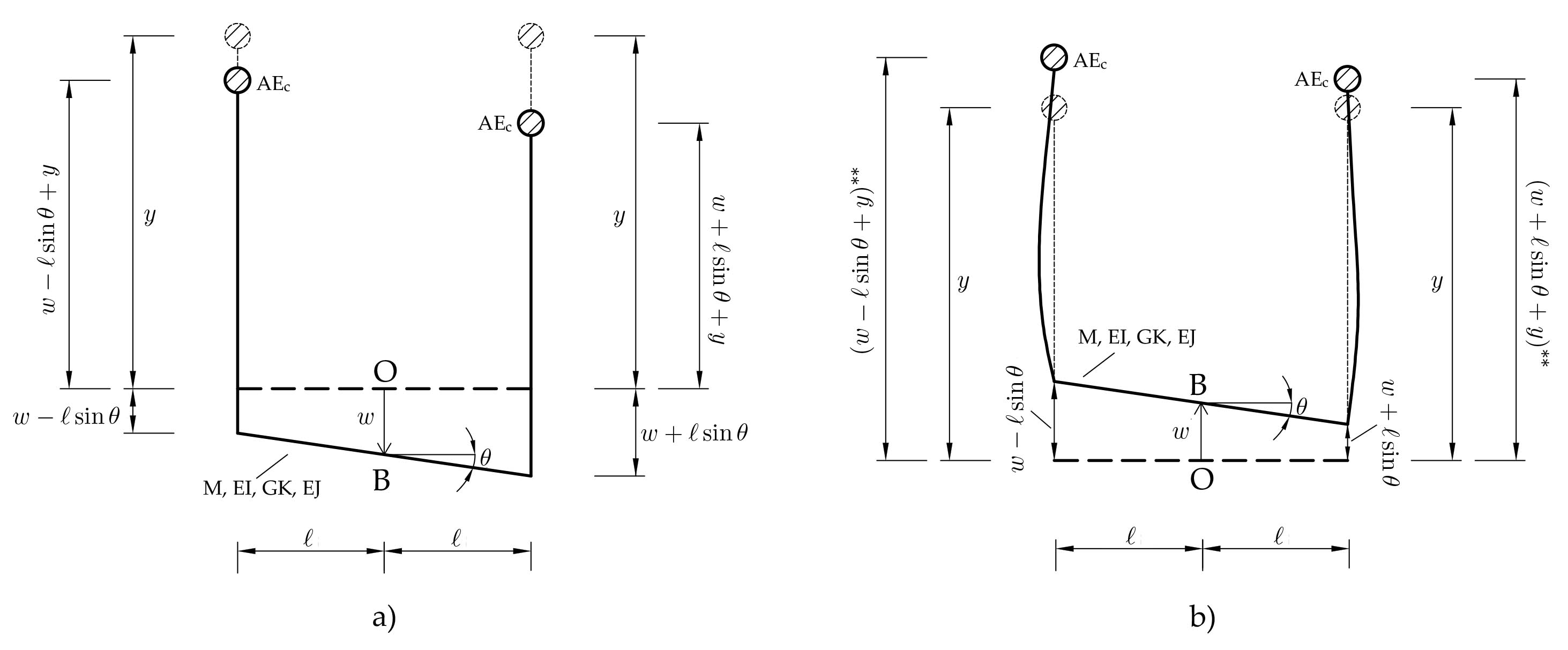}
	\caption{Mutual positions of the cross section of the bridge and of the cables.}
	\label{schema}
\end{figure}
We do not consider the masses of the hangers and of the cables, since they are negligible with respect to the mass of the deck.
The deformations of the deck deriving from bending and torsion are modeled, as for a beam, according to the de Saint Venant and Vlasov theory:
the deck is characterized by the flexural rigidity $EI$, the torsional rigidity $GK$ (by de Saint Venant) and the torsional warping $EJ$
(by Vlasov \cite{Vlasov}). The energy of the deck is given by the kinetic, the gravitational and the deformation contributions, for details see \cite{tesi},
\begin{equation*}
\begin{split}
E_{d}&=\frac{M}{2}\int_{0}^{L}w_{t}^2\hspace{1mm}dx+\frac{M\ell^2}{6}\int_{0}^{L}\theta_{t}^2\hspace{1mm}dx-Mg\int_{0}^{L} w\hspace{1mm}dx\\&+\frac{EI}{2}\int_{0}^{L}w_{xx}^2\, dx+\frac{GK}{2}\int_{0}^{L}\theta_{x}^2\, dx+\frac{EJ}{2}\int_{0}^{L}\theta_{xx}^2\, dx,
\end{split}
\label{en cinetica}
\end{equation*}
where $M$ is the mass linear density of the deck, $\ell$ its semi width, $g$ the gravitational acceleration, $E$ the Young modulus, $G$ the shear modulus, $I$ the moment of inertia, $K$ the torsional constant and $J$ is the warping constant of the cross section. The last term was added by Vlasov \cite{Vlasov} to the de Saint Venant's deformation terms.

\subsection{The deformation energy of the cables}
\label{energia convex}
We assume that the cables have the same mechanical properties and, at rest, they take a parabolic shape given by
\begin{equation}
y(x)=-\dfrac{4f}{L^2}x^2+\dfrac{4f}{L}x-y_0 \hspace{7mm}\forall x\in [0,L],
\label{y(x)}
\end{equation}
where  $y_{0}$ is the height of the towers and $f$ is the cable sag as in Figure \ref{ponte}; for details on the derivation of (\ref{y(x)}) we refer to \cite{biot, falocchi2}.
The local length of the cables is given for all $x\in[0,L]$ by the bounded function
\begin{equation*}
\xi(x):=\sqrt{1+y'(x)^2}, \hspace{15mm}1\leq\xi(x)\leq\xi_M:=\sqrt{1+\bigg(\dfrac{4f}{L}\bigg)^2}.
\label{xi1}
\end{equation*}
From \cite[p.68]{Podolny} we know that in bridge design the sag-span ratio $f/L$ varies between $1/12$ and $1/8$, implying a little variation of $\xi(x)$ on $[0,L]$; indeed, its maximum value, assumed for  $x\in\{0,L\}$, is
$
\xi(0)=\xi(L)=\xi_M\in\big[\frac{\sqrt{10}}{3},\frac{\sqrt{5}}{2}\big]=[1.05, 1.11].
$
 In engineering literature $\xi(x)$ is often approximated with 1, see \cite{biot}. But $\xi(x)$ remains closer to its mean value $\overline{\xi}$ over the interval $[0,L]$; for these reasons we shall use the approximation
\begin{equation}
\xi(x)\approx\overline{\xi}:=\dfrac{\int_{0}^L \xi(x)dx}{L}.
\label{xi}
\end{equation}
We recall that for the TNB, by assuming \eqref{xi}, the maximum error is less than $2.86$\%.
\par
To obtain the energy of the cables we need to find their convexified shapes.
Figure \ref{schema}a) shows the situation with tensioned hangers, in which the edges of the deck have moved downwards of
$w\pm\ell\sin\theta$. In this case, the cables have a convex shape and the hangers behave like inextensible elements so that
the cables have the same displacement of the deck and their positions are $(w\pm\ell\sin\theta+y)$. In Figure \ref{schema}b)
we represent the innovative part of our model. If the endpoints of the cross section of the deck move upwards, above the position
$(w\pm\ell\sin\theta+y)=0$, then the slackening of the hangers may occur, producing a vertical displacement in the cables equal to
$(w\pm\ell\sin\theta+y)^{**}$. The shape of the cables is then given by the convexification of the function $(w\pm\ell\sin\theta+y)$,
that depends on both $x$ and $t$. To determine the deformation energy of a cable we need to compute its variation of length with respect to its
initial length $L_c:=\int_{0}^{L}\sqrt{1+(y')^2}\big)dx$. Then, we introduce the functional
\begin{equation}\label{gamma}
\Gamma: C^1[0,L]\rightarrow \mathbb{R}\, ,\qquad\Gamma(u):=\int_{0}^{L}\big(\sqrt{1+\{[(u+y)^{**}]_{x}\}^2}\big)dx-L_c,
\end{equation}
which is well-defined, since the convexification preserves the $C^1$-regularity of $u$.
The deformation energy $E_C$ of the cables is composed by two contributions. The first is related to the tension at rest and the second to the additional tension due to the increment of the length $\Gamma(w\pm\ell\sin\theta)$ of each cable. Hence if $\overline{\xi}$ is as in \eqref{xi}, we have

\footnotesize
\begin{equation}\label{energy cavi}
\begin{split}
E_{C}=&H\overline{\xi}\bigg[\int_{0}^{L}\bigg(\sqrt{1+\{[(w+\ell\sin\theta+y)^{**}]_{x}\}^2}+\sqrt{1+\{[(w-\ell\sin\theta+y)^{**}]_{x}\}^2}\bigg)\, dx-2L_c\bigg]\\
&+\frac{AE_c}{2L_{c}}\big([\Gamma(w+\ell\sin\theta)]^2+[\Gamma(w-\ell\sin\theta)]^2\big),
\end{split}
\end{equation}\normalsize
with $H=$ horizontal tension, $A=$ sectional area, $E_c=$ Young modulus of the cable.

\subsection{Functional spaces and total energy of the system}

We consider the Hilbert spaces $L^2(0,L)$, $H^1_0(0,L)$ and $H^2\cap H^1_{0}(0,L)$, endowed respectively with the scalar products
\begin{equation*}
(u,v)_{2}=\int_{0}^{L}uv, \hspace{5mm} (u,v)_{H^1}=\int_{0}^{L}u'v',\hspace{5mm} (u,v)_{H^2}=\int_{0}^{L}u''v''.
\end{equation*}
We denote by $H^{*}(0,L)$ the dual space of $H^2\cap H^1_0(0,L)$ with the corresponding duality $\langle \cdot,\cdot\rangle_*$.
The solutions of the equations are required to satisfy $(w,\theta)\in X^2_T$, where
\begin{equation}
X_T:=C^0\big([0,T];H^2\cap H^1_0(0,L)\big)\cap C^1\big([0,T];L^2(0,L)\big)\cap C^2\big([0,T];H^*(0,L)\big).
\label{regularity weak}
\end{equation}

Then, by adding all the energetic contributions of the system, for every $(w,\theta)\in X_T^2$ we find the functional
\footnotesize
\begin{equation}
\begin{split}
\mathcal{E}(w,\theta):=&\int_{0}^{L}\bigg(\frac{M}{2}w_{t}^2\hspace{1mm}+\frac{M\ell^2}{6}\theta_{t}^2\bigg)\hspace{1mm}dx
+\int_{0}^{L}\bigg(\frac{EI}{2}w_{xx}^2+\frac{EJ}{2}\theta_{xx}^2+\frac{GK}{2}\theta_{x}^2\bigg)dx\\
&\hspace{0mm}+H\overline{\xi}\bigg\{\int_{0}^{L}\big(\sqrt{1+\{[(w+\ell\sin\theta+y)^{**}]_{x}\}^2}+\sqrt{1+\{[(w-\ell\sin\theta+y)^{**}]_{x}\}^2}\big)dx\bigg\}\\
&\hspace{0mm}-2H\overline{\xi}L_c+\frac{AE_c}{2L_{c}}\big([\Gamma(w+\ell\sin\theta)]^2+[\Gamma(w-\ell\sin\theta)]^2\big)-Mg\int_{0}^{L}w\hspace{1mm}dx,
\label{energy tot}
\end{split}
\end{equation}
\normalsize
that is well-defined and represents the energy of the system.
\begin{prop}
	The functional $\mathcal{E}: X_T^2\rightarrow \mathbb{R}$ is locally Lipschitz continuous.
	\label{lipschitz functional}
\end{prop}
\begin{proof}
	The statemant holds if, for every bounded subset $X\subset X^2_T$ there exists $\mathcal{L}>0$ such that, given $(w_1,\theta_1)$ and $(w_2,\theta_2)\in X$
we have
	\small
	\begin{equation}
	|\mathcal{E}(w_1,\theta_1)-\mathcal{E}(w_2,\theta_2)|\leq \mathcal{L}\big(||(w_{1}-w_{2})_t||_{1}+||(\theta_{1}-\theta_{2})_t||_{1}+||w_{1}-w_{2}||_{W^{2,1}}+||\theta_{1}-\theta_{2}||_{W^{2,1}}\big).
	\label{lipschitz}
	\end{equation}
	\normalsize
	By (\ref{energy tot}) we observe that the most tricky terms are those including $\Gamma(\cdot)$ and $\overline{\xi}$, while for the others (\ref{lipschitz}) is easily proved. Let us recall the inequality
	\begin{equation*}
	|\sqrt{1+(u_1+v)^2}-\sqrt{1+(u_2+v)^2}|\leq|(u_1+v)-(u_2+v)|=|u_1-u_2|\hspace{5mm}\forall u_1,u_2,v\in\mathbb{R},
	\label{ineq1}
	\end{equation*}
	that gives
	\begin{equation*}
	\begin{split}
	&\big|\sqrt{1+\{[(w_1\pm\ell\sin\theta_1+y)^{**}]_{x}\}^2}-\sqrt{1+\{[(w_2\pm\ell\sin\theta_2+y)^{**}]_{x}\}^2}\big|\\&\hspace{0mm}\leq \big|[(w_1\pm\ell\sin\theta_1+y)^{**}-(w_2\pm\ell\sin\theta_2+y)^{**}]_x\big|.
	\end{split}
	\label{ineq}
	\end{equation*}
	Hence, it is possible to use \eqref{f:th2} so that there exists $L_1>0$ such that
	\small
	\begin{equation*}
	\begin{split}
	&H\overline{\xi}\int_{0}^{L}\big|\sqrt{1+\{[(w_1\pm\ell\sin\theta_1+y)^{**}]_{x}\}^2}-\sqrt{1+\{[(w_2\pm\ell\sin\theta_2+y)^{**}]_{x}\}^2}\big|dx\\&\leq H\overline{\xi}\big(||(w_1-w_2)_x||_1+\ell||(\sin\theta_1-\sin\theta_2)_x||_1\big)\leq L_1\big(||(w_1-w_2)_x||_1+||\theta_1-\theta_2||_{W^{1,1}}\big).
	\end{split}
	\end{equation*}
	\normalsize
	The same argument can be applied to the terms $[\Gamma(w\pm\ell\sin\theta)]^2$, see (\ref{gamma}).
\end{proof}
This result enables us to use the notion of Clarke subdifferential \cite{clarke} and to compute the variation of (\ref{energy tot}) in the general framework of the differential inclusions. We also point out that some problems in elasticity may be tackled with a slightly different notion of nonsmooth critical points,
see \cite{degiovanni}.

\section{Suspension bridges with convexified cables}
\label{The variational problem}

\subsection{The variation of the deformation energy of the cables}
\label{The variation of the deformation energy of the cables}
The presence of the convexified functions within the functional $\mathcal{E}(w,\theta)$ in (\ref{energy tot}) introduces some difficulties in computing
its variation; from Proposition \ref{p:prop variaz} the unilateral Gateaux derivative exists and is always bounded, while the Gateaux derivative may not exist in some cases. Let us focus on one cable, the other being similar. We introduce

\small \[D^-:= \big[H\overline{\xi}+\frac{AE_c}{L_{c}}\Gamma(w+\ell\sin\theta)\big]\int_{0}^{L}\frac{[(w+\ell\sin\theta+y)^{**}]_{x}(\mathcal{J}^\varphi_-)'}{\sqrt{1+\{[(w+\ell\sin\theta+y)^{**}]_{x}\}^2}}
\hspace{1mm}dx\]
\[D^+:= \big[H\overline{\xi}+\frac{AE_c}{L_{c}}\Gamma(w+\ell\sin\theta)\big]\int_{0}^{L}\frac{[(w+\ell\sin\theta+y)^{**}]_{x}(\mathcal{J}^\varphi_+)'}{\sqrt{1+\{[(w+\ell\sin\theta+y)^{**}]_{x}\}^2}}
\hspace{1mm}dx,\]
\normalsize
 where $\mathcal{J}^\varphi_\pm(x)$ are defined in (\ref{zbar2}) with $f=(w+\ell\sin\theta+y)$. By applying Proposition \ref{p:prop variaz}, we find the following inclusion related to the variation of the energy (\ref{energy cavi}) with respect to $w$ \[\langle dE_C(w,\theta),\varphi\rangle\in \big[\min\{D^-,D^+\},\max\{D^-,D^+\}\big].\]
To avoid this heavy notation, in the sequel we always write
\footnotesize
\begin{equation*}
\begin{split}
\langle dE_C(w,\theta),\varphi\rangle\in
&\big[H\overline{\xi}+\frac{AE_c}{L_{c}}\Gamma(w+\ell\sin\theta)\big]\int_{0}^{L}\frac{[(w+\ell\sin\theta+y)^{**}]_{x}(\mathcal{J}^\varphi_\pm)'}{\sqrt{1+\{[(w+\ell\sin\theta+y)^{**}]_{x}\}^2}}
\hspace{1mm}dx\hspace{5mm}\forall \varphi\in C^{\infty}_c(\mathcal{I}).
\end{split}
\label{variation w energy}
\end{equation*}
\normalsize
By applying Theorem \ref{prop variaz gen} with $\Lambda(\theta)=\sin\theta$
we obtain the inclusion related to the variation of the energy (\ref{energy cavi}) with respect to $\theta$
\footnotesize
\begin{equation*}
\begin{split}
&\langle dE_C(w,\theta),\psi\rangle
\in\big[H\overline{\xi}+\frac{AE_c}{L_{c}}\Gamma(w+\ell\sin\theta)\big]\ell\int_{0}^{L}\frac{[(w+\ell\sin\theta+y)^{**}]_{x}(\mathcal{G}^{\theta,\psi}_\pm)_x}{\sqrt{1+\{[(w+\ell\sin\theta+y)^{**}]_{x}\}^2}}\hspace{1mm}dx \hspace{3mm}\forall \psi\in C^{\infty}_c(\mathcal{I}),
\end{split}
\label{variation teta energy Hc}
\end{equation*}
\normalsize
where $\mathcal{G}^{\theta,\psi}_\pm(x)$ is defined for every $t\geq0$ as follows: given $g_i^\pm\colon K^i \to \overline{\mathbb{R}},$ with $\ i\in J_C$
\footnotesize
\begin{equation}
g_i^\pm(x) :=
\begin{cases}
\psi(x)\cos\big(\theta(x)\big) & x \in K^i \cap (K_f \cup \{a,b\}),\\
\pm\infty & x \in K^i \setminus (K_f \cup \{a,b\}).
\end{cases}\quad
\mathcal{G}^{\theta,\psi}_\pm(x):=
\begin{cases}
\pm(\pm g_i^\pm)^{**}(x) & x \in K^i,\\
\psi(x)\cos\big(\theta(x)\big) & x \in \widetilde{K}_f.
\end{cases}
\label{G theta pm}
\end{equation}
\normalsize

Note that the functions $\mathcal{J}^\varphi_\pm$ and $\mathcal{G}^{\theta,\psi}_\pm$ are spatially continuous with a finite number of angular points, so that $(\mathcal{J}^{\varphi}_\pm)'$ and $(\mathcal{G}^{\theta,\psi}_\pm)_x$ are bounded on the interval $[0,L]$ and continuous almost everywhere in $[0,L]$, see Remark \ref{rmk 2.5}. \par
In the simple cases in which the cable is strictly convex (or concave!) we gain the differentiability of (\ref{energy tot}) and the inclusions become equalities. In the first case, because $\widetilde{K}_f=\mathcal{I}$ so that $\mathcal{J}^\varphi_\pm$ and $\mathcal{G}^{\theta,\psi}_\pm$ coincide respectively with $\varphi$ and $\psi\cos\theta$. In the second case, $K^1=\overline{I}^1=\overline{\mathcal{I}}$ so that $\mathcal{J}^\varphi=\mathcal{G}^{\theta,\psi}=0$ and $(w+\ell\sin\theta+y)^{**}=-y_0$ for all $x\in[0,L]$; this situation corresponds to a zero variation in the cable energy since the slackening of all the hangers occurs, implying the total disconnection between the cable and the deck. We point out that in the case where the cable is perfectly horizontal we obtain the same physical result, due to $(w+\ell\sin\theta+y)_x=[(w+\ell\sin\theta+y)^{**}]_x=0$ for all $x\in[0,L]$, while $\mathcal{I}^\varphi_\pm$ and $\mathcal{G}^{\theta,\psi}_\pm$ maintain their oscillatory nature.

\subsection{The system of partial differential inclusions}\label{The system of partial differential inclusions}

We set here the problem in the general framework of partial differential inclusions, resulting from the variation of (\ref{energy tot}). The
subscripts $\alpha$ and $\beta$ denote the terms corresponding respectively to the cable with shape $(w+\ell\sin\theta+y)^{**}$ and
$(w-\ell\sin\theta+y)^{**}$. Hence, we have $\mathcal{J}^\varphi_{\alpha\pm}(x)$, $\mathcal{G}^{\theta,\psi}_{\alpha\pm}(x)$ and $\mathcal{J}^\varphi_{\beta\pm}(x)$, $\mathcal{G}^{\theta,\psi}_{\beta\pm}(x)$ that correspond to $\mathcal{J}^\varphi_{\pm}(x)$, $\mathcal{G}^{\theta,\psi}_{\pm}(x)$ related respectively to $f_\alpha=(w+\ell\sin\theta+y)$ and $f_\beta=(w-\ell\sin\theta+y)$, as defined in
\eqref{zbar2} and \eqref{G theta pm}. Moreover, we include all the nonlinearities into the functionals

\small
\begin{equation*}
\begin{split}
&h_\alpha(Z,\Theta):=-\bigg(H\overline{\xi}+\frac{AE_c}{L_{c}}\Gamma(Z+\ell\sin\Theta)\bigg)\dfrac{[(Z+\ell\sin\Theta+y)^{**}]_{x}}{\sqrt{1+\{[(Z+\ell\sin\Theta+y)^{**}]_{x}\}^2}}\\&h_\beta(Z,\Theta):=-\bigg(H\overline{\xi}+\frac{AE_c}{L_{c}}\Gamma(Z-\ell\sin\Theta)\bigg)\dfrac{[(Z-\ell\sin\Theta+y)^{**}]_{x}}{\sqrt{1+\{[(Z-\ell\sin\Theta+y)^{**}]_{x}\}^2}}.
\end{split}
\label{h1,h2}
\end{equation*}
\normalsize\par As for the action, one has to take the difference between kinetic energy and potential energy and integrate over an interval of time $[0, T]$:
\footnotesize
\begin{equation*}
\begin{split}
\mathcal{A}(w,\theta):=&\int_0^T\bigg[\int_{0}^{L}\bigg(\frac{M}{2}w_{t}^2\hspace{1mm}+\frac{M\ell^2}{6}\theta_{t}^2\bigg)\hspace{1mm}dx
-\int_{0}^{L}\bigg(\frac{EI}{2}w_{xx}^2+\frac{EJ}{2}\theta_{xx}^2+\frac{GK}{2}\theta_{x}^2\bigg)dx\\
&\hspace{0mm}-H\overline{\xi}\bigg\{\int_{0}^{L}\big(\sqrt{1+\{[(w+\ell\sin\theta+y)^{**}]_{x}\}^2}+\sqrt{1+\{[(w-\ell\sin\theta+y)^{**}]_{x}\}^2}\big)dx-2L_c\bigg\}\\
&\hspace{0mm}-\frac{AE_c}{2L_{c}}\big([\Gamma(w+\ell\sin\theta)]^2+[\Gamma(w-\ell\sin\theta)]^2\big)+Mg\int_{0}^{L}w\hspace{1mm}dx\bigg]dt.
\end{split}
\label{action}
\end{equation*}
\normalsize
The differential inclusion describing the motion of the bridge is obtained by considering the critical points of the functional $\mathcal{A}$, which leads to the following
\begin{definition}
	We say that $(w,\theta)\in X^2_T$, see \eqref{regularity weak}, is a weak solution of the differential inclusion, resulting from critical points of the action $\mathcal{A}$, if $(w,\theta)$ satisfies
	\footnotesize
	\begin{equation}
	\begin{cases}
	M\langle w_{tt},\varphi \rangle_{*}+EI(w,\varphi)_{H^2}-\big( Mg,\varphi\big)_2\in\bigg(h_\alpha(w,\theta),(\mathcal{J}^\varphi_{\alpha\pm})'\bigg)_2+\bigg(h_\beta(w,\theta),(\mathcal{J}^\varphi_{\beta\pm})'\bigg)_2,\\
	 \dfrac{M\ell}{3}\langle\theta_{tt},\psi\rangle_*+\dfrac{EJ}{\ell}(\theta,\psi)_{H^2}+\dfrac{GK}{\ell}(\theta,\psi)_{H^1}\in\bigg(h_\alpha(w,\theta),(\mathcal{G}^{\theta,\psi}_{\alpha\pm})_x\bigg)_2-\bigg(h_\beta(w,\theta),(\mathcal{G}^{\theta,\psi}_{\beta\pm})_x\bigg)_2,
	\end{cases}
	\label{eq weak sist2}
	\end{equation}
	\normalsize
	for all $\varphi, \psi\in H^2\cap H^1_0(0,L)$ and $t>0$.
\end{definition}
The system (\ref{eq weak sist2}) is complemented with the initial conditions:
\begin{equation}
\begin{split}
&w(x,0)=w^0(x),\hspace{4mm} \theta(x,0)=\theta^0(x) \hspace{25mm} {\rm for } \;x\in (0,L)\\
&w_{t}(x,0)=w^1(x),\hspace{4mm} \theta_{t}(x,0)=\theta^1(x) \hspace{23mm} {\rm for } \;x\in (0,L),\\
\end{split}
\label{ics}
\end{equation}
having the following regularity
\begin{equation}
w^0, \theta^0\in H^2\cap H^1_0(0,L), \hspace{5mm}w^1,\theta^1\in L^2(0,L).\\
\label{regularity ics}
\end{equation}

From \cite{cellina1, cellina2, clarke} we learn that existence results for a differential inclusion can be a difficult task, requiring some regularity
of the right hand side terms, e.g.\ the continuity. For our purposes, to approach problem (\ref{eq weak sist2})-(\ref{ics}) in fully generality is
not necessary: since we deal with a civil structure, we perform some simplifications. We follow a suggestion from \cite[p.23]{bleich} which says that
{\em ``\dots out of the infinite number of possible modes of motion in which a suspension bridge might vibrate, we are interested only in a few, to wit: the ones having the smaller numbers of loops or half waves''}. Indeed, civil structures typically oscillate on low modes since higher modes do not appear in
real situations due to large bending energy, see \cite{bleich}. This suggestion mathematically corresponds to project an infinite dimensional space
on a finite dimensional subspace, using the Galerkin approximation.
Whence, we take $\{e_k\}^{\infty}_{k=1}$ as an orthogonal basis of $L^2(0,L)$, $H^1_{0}(0,L)$, $H^2\cap H^1_{0}(0,L)$, where
\begin{equation*}
e_{k}(x)=\sqrt{\frac{2}{L}}\sin\bigg(\frac{k\pi x}{L}\bigg), \hspace{5mm}|| e_{k}||_{2}=1,\hspace{5mm}||e_{k}||_{H^1}=\frac{k\pi}{L},\hspace{5mm}|| e_{k}||_{H^2}=\frac{k^2\pi^2}{L^2},
\end{equation*}
and, for any $n\geq 1$, we introduce the space $E_n:={\rm span} \{e_1,\dots,e_n\}$. For any $n\geq 1$ we seek a couple $(w_n,\theta_n)\in X_T^2$ such that
\begin{equation*}
w_n(x,t)=\sum_{k=1}^{n} w_n^k(t) \hspace{1mm}e_k(x), \hspace{5mm} \theta_n(x,t)=\sum_{k=1}^{n} \theta_n^k(t) \hspace{1mm}e_k(x),
\label{finito dim}
\end{equation*}
and satisfying (\ref{eq weak sist2}) only for the test functions $\varphi, \psi\in E_n$, thereby obtaining a finite system of ordinary differential inclusions. In fact, in this finite dimensional setting, the inclusions become equalities, since for every fixed $n\in\mathbb{N}$, all the intervals of affinity (if any) of $(w_n\pm\ell\sin\theta_n+y)^{**}$ are such that $(w_n\pm\ell\sin\theta_n+y)>(w_n\pm\ell\sin\theta_n+y)^{**}$. Then Corollary \ref{prop variaz} applies and the Gateaux derivative exists, leading to a finite system of ordinary differential equations ($r=1,\dots,n$) with initial conditions

\footnotesize
\begin{equation}\label{eq weak sist gal}
\begin{cases}
M\big((w_n)_{tt},e_r\big)_2+EI\big(w_n,e_r\big)_{H^2}-(Mg,e_r)_2=\\\hspace{25mm}\bigg(h_\alpha (w_n,\theta_n),[\mathcal{J}^{e_r}_\alpha]'\bigg)_2+\bigg(h_\beta(w_n,\theta_n),[\mathcal{J}^{e_r}_\beta]'\bigg)_2\vspace{3mm}\\
\dfrac{M\ell}{3}\big((\theta_n)_{tt},e_r\big)_2+\dfrac{EJ}{\ell}\big(\theta_n,e_r\big)_{H^2}+\dfrac{GK}{\ell}\big(\theta_n,e_r\big)_{H^1}=\\\hspace{25mm}\bigg(h_\alpha (w_n,\theta_n),[\mathcal{G}^{\theta_n,e_r}_\alpha]_x\bigg)_2-\bigg(h_\beta (w_n,\theta_n),[\mathcal{G}^{\theta_n,e_r}_\beta]_x\bigg)_2,\\
w^k_n(0)=(w^0,e_k)_2, \hspace{5mm}\theta^k_n(0)=(\theta^0,e_k)_2\hspace{5mm}
\dot{w}^k_n(0)=(w^1,e_k)_2,\hspace{5mm} \dot{\theta}^k_n(0)=(\theta^1,e_k)_2.
\end{cases}
\end{equation}
\normalsize

In Section \ref{Proof of the Theorem 2} we prove
\begin{teo}
	Let $n\geq 1$ an integer and $T>0$ (possibly $T=\infty$), then for all $w^0, \theta^0, w^1,\theta^1$ satisfying (\ref{regularity ics}) there exists a unique solution $(w_n,\theta_n)\in X_T^2$ of \eqref{eq weak sist gal}.
	\label{teo1}
\end{teo}

This justifies the following

\begin{definition}\label{defapproximable}
For all $n\in\mathbb{N}$, we say that the solution $(w_n,\theta_n)$ of \eqref{eq weak sist gal} is an \textbf{approximate solution} of (\ref{eq weak sist2})-(\ref{ics}).
We say that $(w,\theta)\in X_T^2$ is an \textbf{approximable solution} of (\ref{eq weak sist2})-(\ref{ics}) if there exists a sequence of approximate solutions of (\ref{eq weak sist2})-(\ref{ics}), converging to it as $n\rightarrow\infty$, up to a subsequence.
\end{definition}

We now state the main result of this section, whose proof is given in Section \ref{Proof of the Theorem 2}.

\begin{teo}
	Let $T>0$ (possibly $T=\infty$), then for all $w^0, \theta^0, w^1,\theta^1$ satisfying (\ref{regularity ics}) there exists an approximable solution of (\ref{eq weak sist2}) which satisfies (\ref{ics}) on $[0,T]$.
	\label{corol inclusione}
\end{teo}

\begin{remark}
	We refer to \cite{tesi} for some consequences in existence and uniqueness results when we consider the same problem with variable $\xi(x)$, i.e. not assuming \eqref{xi}.	The results obtained on $(w_n,\theta_n)$ can be achieved in the same way considering a different number of modes
	for $w$ and $\theta$, i.e.\ taking $(w_n, \theta_\nu)$ with $n\neq \nu$.
\end{remark}


\section{Numerical results}
\label{Numerical analysis}

In this section we present some numerical experiments on the system \eqref{eq weak sist gal}. The results are obtained with the software Matlab$^{\textregistered}$, adopting the mechanical constants of the TNB in Table \ref{tab mec}.
\begin{table}[h]
	\begin{center}
		\begin{tabular}{*{3}{c}}
			\hline
			$E$:&200\hspace{1mm}000MPa& Young modulus of the deck (steel)\\
			$E_c$:&185\hspace{1mm}000MPa& Young modulus of the cables (steel)\\
			$G$:&81\hspace{1mm}000MPa& Shear modulus of the deck (steel)\\
			$L$:&853.44m& Length of the main span\\
			$\ell:$&6m& Half width of the deck\\
			$f$:&70.71m& Sag of the cable\\
			$I$:&0.154m$^4$& Moment of inertia of the deck cross section\\
			$K$:&6.07$\cdot10^{-6}$m$^4$& Torsional constant of the deck\\
			$J$:&5.44m$^6$& Warping constant of the deck\\
			$A$:&0.1228m$^2$& Area of the cables section\\
			$M$:&7198kg/m&Mass linear density of the deck \\
			$H$:&45\hspace{1mm}413kN& Initial tension in the cables\\
			$L_c$:&868.815m& Initial length of the cables, see (\ref{gamma})\\
			\hline
		\end{tabular}
	\end{center}
	\caption{TNB mechanical constants.}
	\label{tab mec}
\end{table}

When we speak about a mode like $\sin\big(\frac{k\pi }{L}x\big)$, we refer to a motion with $k-1$ nodes, in which the latter are the zeros of the sine function in $(0,L)$. Let us recall some meaningful witnesses that led our modeling choices. From
\cite[p.28]{TNB} we know that for the TNB {\em ``seven different motions have been definitely identified on the main span of the bridge''}.
The morning of the failure Farquharson, a witness of the collapse described a torsional motion like $\sin\big(\frac{2\pi }{L}x\big)$, writing
\cite[V-2]{TNB} {\em ``a violent change in the motion was noted. [\dots] the motions, which a moment before had involved a number of waves (nine or ten) had shifted almost instantly to two [\dots] the node was at the center of the main span and the structure was subjected to a violent torsional action about this point''}.\par
By Theorem \ref{corol inclusione} we may consider an approximable solution of \eqref{eq weak sist2} and decide how many modes to include in the finite
dimensional approximation. From \cite{TNB} we learn that, at the TNB, oscillations with more than 10 nodes on the main span were never seen. Hence,
we consider the first 10 longitudinal modes and the first 4 torsional modes; this is a good compromise between limiting
computational burden and focusing on the instability phenomena visible at TNB. Further experiments with a larger number of modes did not highlight
significant changes in the instability thresholds.
Given the boundary conditions, we seek solutions of (\ref{eq weak sist gal}) in the form
\begin{equation}
w(x,t)=\sum_{k=1}^{10} w_k(t) \hspace{1mm}e_k, \hspace{5mm} \theta(x,t)=\sum_{k=1}^{4} \theta_k(t) \hspace{1mm}e_k
\label{approx}
\end{equation}
where $e_{k}(x)=\sqrt{\dfrac{2}{L}}\sin\bigg(\dfrac{k\pi x}{L}\bigg)$ and  $\sqrt{\dfrac{2}{L}}$ is a pure number with no unit of measure;
we call $\overline{w}_k(t):=\sqrt{\frac{2}{L}}w_k(t)$ the \textbf{k-th longitudinal mode} and $\overline{\theta}_k(t):=\sqrt{\frac{2}{L}}\theta_k(t)$ the
\textbf{k-th torsional mode}. We obtain a system of 14 ODEs as (\ref{eq weak sist lin gal  2n}) with initial conditions
\begin{equation*}
\begin{split}
&w_k(0)=w^0_k=(w^0,e_k)_2, \hspace{10mm}\dot{w}_k(0)=w^1_k=(w^1,e_k)_2,\hspace{5mm}\forall k=1,\dots,10\\&\theta_k(0)=\theta^0_k=(\theta^0,e_k)_2,\hspace{13mm}
\dot{\theta}_k(0)=\theta^1_k=(\theta^1,e_k)_2\hspace{9.5mm}\forall k=1,\dots,4.
\end{split}
\end{equation*}
We put $\overline{w}_k^0:=\sqrt{\frac{2}{L}}w_k^0$, $\overline{w}_k^1:=\sqrt{\frac{2}{L}}w_k^1$ and similarly for the $\theta$ initial conditions. Since we study an isolated system we assume that on the bridge there is a balance between damping and wind on an interval $[0,T]$ for sufficiently small $T>0$.
We consider a time lapse of $[0,120s]$, which is a small time interval compared to 70 minutes of violent oscillations recorded prior to the TNB collapse \cite{TNB}, enough to see the possible sudden transfer of energy between modes. We study the system during its steady motion, in which the oscillation of a $j$-th longitudinal mode prevails, and we perturb all the other modes with an initial condition 10$^{-3}$ smaller, i.e.\ in dimensionless form
\begin{equation*}
\begin{split}
&\overline{w}_{k}^0=10^{-3}\cdot \overline{w}_{j}^0\hspace{5mm}\forall k\neq j,\hspace{15mm}\overline{\theta}_{k}^0=\overline{w}_{k}^1=\overline{\theta}_{k}^1=10^{-3}\cdot \overline{w}_{j}^0\hspace{5mm}\forall k.
\end{split}
\label{ic22}
\end{equation*}
Following this approach we say that the initial energy of our system corresponds to that of the longitudinal mode excited and represents, indirectly, the wind energy introduced on the bridge through the so-called vortex shedding (see e.g.\ \cite{wang,xin,zampoli} and the monograph \cite{Paidoussis}), although in
the present paper the energy will be inserted through the initial conditions in a conservative system.\par
According to the Report \cite[p.20]{TNB}, in the months prior to the collapse, {\em ``one principal mode of oscillation prevailed and the modes of
	oscillation frequently changed''}. Therefore, we follow \cite{garrione} and we consider that the approximate solution \eqref{approx} has an initially
prevailing longitudinal mode, that is, there exists $j=1,...,10$ such that $w_j(0)$ is much larger than all the other initial data (both longitudinal and torsional). Then the $j$-th longitudinal mode is torsionally stable if all the torsional components $\theta_{k}(t)$ remain small for all $t$. In our analysis we aim to be more precise and we give a {\em quantitative} characterization of ``smallness''.
We consider thresholds of instability following \cite{garrione} and we say that the $j$-th longitudinal mode is \textbf{torsionally unstable} if at least one torsional mode grows about 1 order in amplitude in the time lapse $[0,120s]$. From a numerical point of view we define the \textbf{threshold of instability} of the j-th longitudinal mode excited as
\begin{equation}
W^0_{j}:=\bigg\{\inf \overline{w}^0_j:\hspace{2mm}\max\limits_{k}\bigg\{\max\limits_{t\in[0,T]}|\overline{\theta}_{k}(t)|\bigg\}\geq 10^{-2}\cdot\overline{w}_{j}^0\bigg\};
\label{soglia}
\end{equation}
this condition allows us to obtain thresholds $W^0_j$ accurate enough for our purposes.\par
As explained in Section \ref{energia convex}, through the convexification procedure, we are able to simulate the slackening of the hangers. To measure the slackening quantity occurring in our numerical experiments, we
identify the two cables by the subscripts $\alpha$ and $\beta$ as in Section \ref{The system of partial differential inclusions}, and we recall that, in the numerical discretization, $[0,T]$ is equally divided in $m$ time steps; for each time step we
compute $\Delta t_h$ ($h=1,\dots,m$), a measure of the percentage of slacken hangers, i.e.\ the ratio between the measure of the union of the intervals
of linearity for each cable and the length of the deck $L$:
\begin{equation*}
\mathcal{M}^\alpha_h:=\dfrac{1}{L}\left|\bigcup\limits_{i\in J_\alpha}I^i_{\alpha}\right|\hspace{10mm}
\mathcal{M}^\beta_h:=\dfrac{1}{L}\left|\bigcup\limits_{j\in J_\beta}I^j_{\beta}\right|\hspace{10mm}\forall h=1,\dots,m.
\end{equation*}
Since the angle of rotation is small, the two cables behave quite similarly and, therefore, we define a mean value of the \textbf{measure of slackening} as
\begin{equation}
\mathcal{M}=\dfrac{1}{2m}\bigg[\sum\limits_{h=1}^{m}\mathcal{M}^\alpha_h+\sum\limits_{h=1}^{m}\mathcal{M}^\beta_h\bigg].
\label{misura}
\end{equation}

Our purpose is to compare the instability thresholds of the model with convexification to those of the same model without convexification,
see \cite{falocchi2}, i.e.\ we study how the slackening of the hangers affects the system. In Table \ref{tab num} we have this comparison in terms of initial energy and amplitude threshold of instability of the $j$-th longitudinal mode excited, computed following (\ref{soglia}).
For each numerical experiment we verified the energy conservation, ascertaining a relative error, $|(\max \mathcal{E}(t)-\min \mathcal{E}(t))/\mathcal{E}(0)|$, on the integration time $[0,120s]$, less than $4\cdot 10^{-3}$.\par
\begin{table}[h]\centering
	\begin{tabular}{|c|c|c|c|c|c|}
		\hline
		&\multicolumn{3}{c|}{\textbf{Convexification}}&\multicolumn{2}{c|}{\textbf{No convexification}}\\
		&\multicolumn{3}{c|}{(Slackening)}&\multicolumn{2}{c|}{(Rigid hangers)}\\
		\cline{2-6}
		Mode&$W^0_{j}$[m]&$\mathcal{E}(0)$[J]&$\mathcal{M}$[\%]&$W^0_{j}$[m]&$\mathcal{E}(0)$[J]\\
		\hline
		1&4.09&7.96$\cdot 10^7$&1.92&4.09&7.96$\cdot 10^7$\\
		\hline
		2&8.37&8.74$\cdot 10^7$&2.94&8.22&8.37$\cdot 10^7$\\\hline
		3&4.89&8.58$\cdot 10^7$&2.40&4.82&8.23$\cdot 10^7$\\\hline
		4&5.35&1.63$\cdot 10^8$&41.79&4.92&1.35$\cdot 10^8$\\\hline
		5&4.25&1.77$\cdot 10^8$&39.40&3.93&1.52$\cdot 10^8$\\\hline
		6&3.64&1.64$\cdot 10^8$&43.46&2.64&8.72$\cdot 10^7$\\\hline
		7&3.65&2.38$\cdot 10^8$&51.72&5.25&8.29$\cdot 10^8$\\\hline
		8&3.28&2.27$\cdot 10^8$&50.05&5.15&1.12$\cdot 10^9$\\\hline
		9&2.31&1.54$\cdot 10^8$&42.55&3.87&7.40$\cdot 10^8$\\\hline
		10&2.65&2.34$\cdot 10^8$&52.73&3.41&6.97$\cdot 10^8$\\
		\hline
	\end{tabular}
	\caption{Thresholds of instability as in (\ref{soglia}), corresponding energy and measure of slackening as in (\ref{misura}), varying the longitudinal mode excited on $[0,120s]$.}
	\label{tab num}
\end{table}
From the data in Table \ref{tab num} we notice different tendencies depending on the mode excited. The first 3 longitudinal modes give substantially the same thresholds of instability in the case with convexification and without, due to the very low percentage of slackening, see $\mathcal{M}$. This fact is not surprising, since a longitudinal motion of the deck like a $\sin(\frac{\pi}{L}x)$ modifies the convexity of the cable only for very large displacements, requiring a so large amount of energy that the threshold of instability is achieved before the appearance of slackening.\par
\begin{figure}[htbp]
	\centering
	\includegraphics[width=13cm]{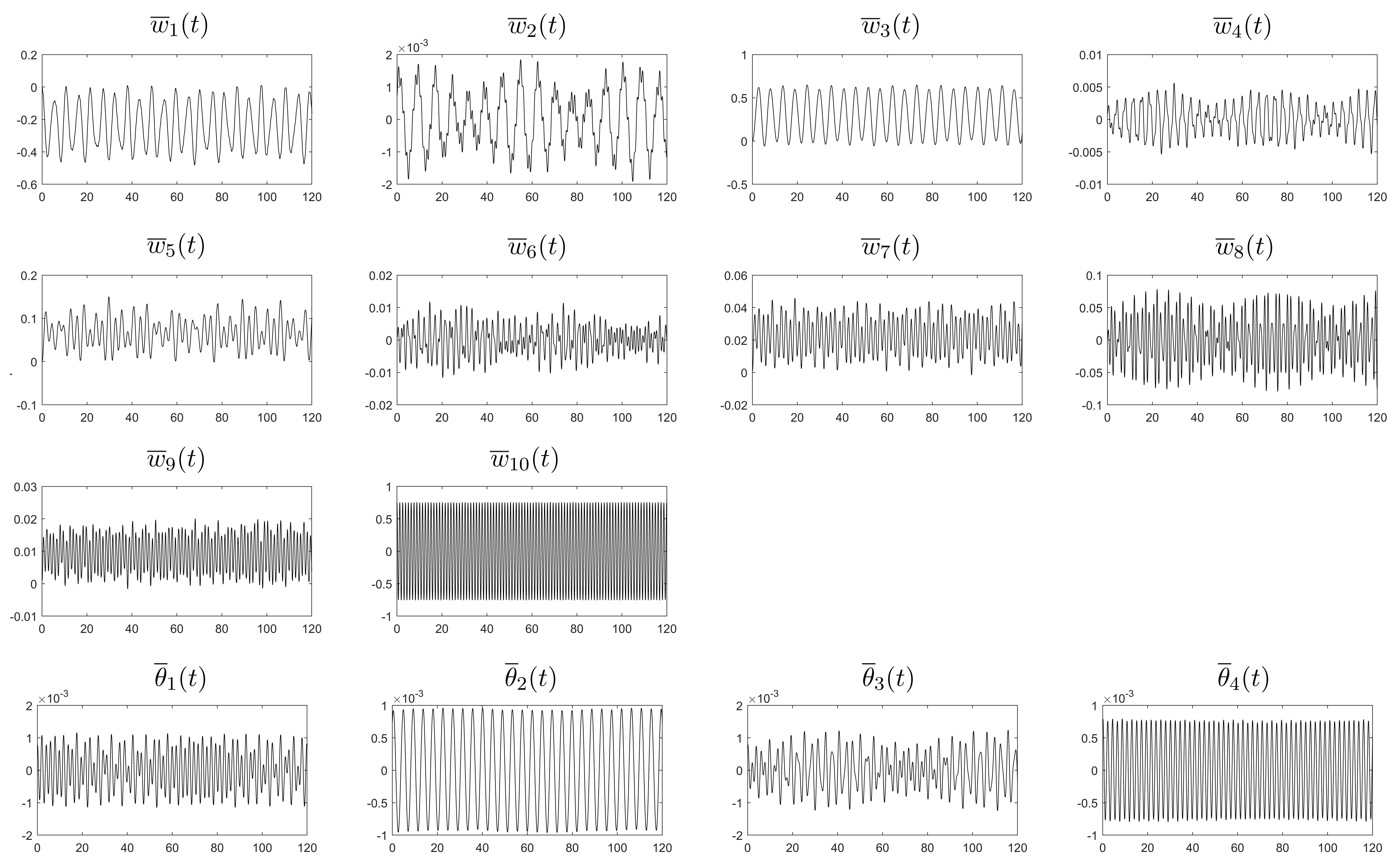}
	\caption{Plots of $\overline{w}_k(t)$ ($k=1,\dots,10$) in meters and $\overline{\theta}_k(t)$ ($k=1,\dots,4$) in radians on $[0,120s]$ with $\overline{w}_{10}^0$=0.75m.}
	\label{analisi stab}
\end{figure}
Quite different is the behavior of the modes from the 4$^{th}$ onward, since in these cases we appreciate differences between two models. We distinguish two tendencies respectively for the intermediate modes (4$^{th}$, 5$^{th}$ and 6$^{th}$) and the higher modes. The thresholds of the intermediate modes reveal that the instability arises earlier for the model with inextensible hangers, so that the latter can be adopted in favor of safety. We point out that the 4$^{th}$, 5$^{th}$ and 6$^{th}$ modes were not seen the day of the collapse of the TNB; the witnesses recorded that, before the rise of the torsional instability, the bridge manifested longitudinal oscillations with 9 or 10 waves, involving the motion of higher longitudinal modes. For these modes the presence of the slackening puts down the thresholds of instability so that the assumption of rigid hangers is not in favor of safety. We underline that in these cases we see more instability despite the injection of energy is smaller; this behavior is peculiar of the hangers slackening that favors a greater transfer of energy between modes with respect to the case with inextensible hangers, see also \cite{falocchi}.
The results in Table \ref{tab num} highlight furthermore that the 9$^{th}$ and $10^{th}$ longitudinal modes present the lowest torsional instability threshold in the case with slackening, confirming the real observations at the TNB collapse.\par
In Figure \ref{analisi stab} we exhibit an example of stability obtained on the system with convexification, imposing $\overline{w}_{10}^0=$0.75m; we notice a very little exchange of energy between modes and, in general, the torsional modes oscillate around their initial amplitude, revealing a stable behavior. In this case some slack is present ($\mathcal{M}=13.50\%$), while reducing further the initial amplitude, e.g.  $\overline{w}_{9}^0\leq$0.60m or $\overline{w}_{10}^0\leq$0.55m, would produce the total absence of slackening and a clear stable situation, see \cite{falocchi2}. \par
For brevity in Figure \ref{analisi instab} we present only the torsional modes related to the instability thresholds of the 9$^{th}$ and 10$^{th}$ longitudinal modes, obtained respectively applying $\overline{w}_{9}^0=$2.31m and $\overline{w}_{10}^0=$2.65m. In general, when (\ref{soglia}) is verified all the torsional modes begin to grow, but Figure \ref{analisi instab} confirms that the 9$^{th}$ and 10$^{th}$ longitudinal modes are more prone to develop instability on the 2$^{nd}$ torsional mode, since it attains the largest growth on $[0,120s]$.\par
\begin{figure}[htbp]
	\centering
	\includegraphics[width=13cm]{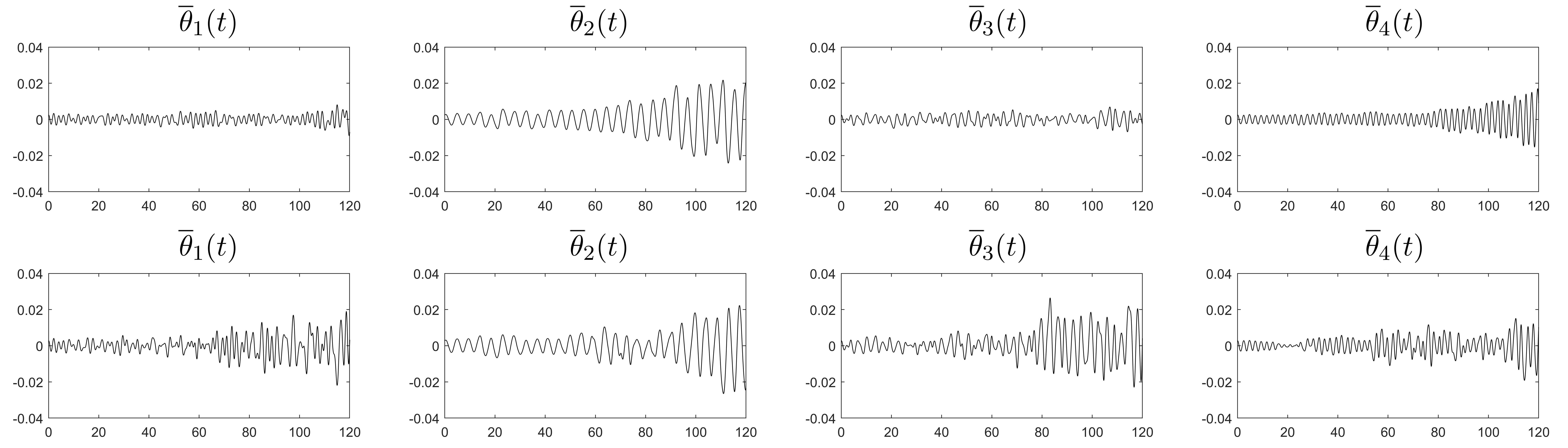}
	\caption{$\overline{\theta}_k(t)$ ($k=1,\dots,4$) in rad. on $[0,120s]$ with $\overline{w}_9^0$=2.31m (above) and $\overline{w}_{10}^0$=2.65m (below).}
	\label{analisi instab}
\end{figure}
Our numerical results show that structures displaying only low modes of vibration may be treated assuming inextensible hangers; this simplification reduces the computational costs and gives safe instability thresholds. On the other hand, if the structure vibrates on higher modes, this assumption could give overestimated thresholds to the detriment of safety; in this case the slackening of the hangers plays an important role. This fact should be a warning for the designers of bridges that are able to exhibit, in realistic situations, large vibration frequencies.

\section{Proofs of the results on the convexification}
\label{proof prop3}

The proofs of the results of Section~\ref{s:convex}
require some basic tools of convex analysis (see e.g.\ \cite{temam2,Rock}).
Given a closed convex set $E\subset\mathbb{R}^n$,
a point $p \in \partial E$ is an extreme point of $E$
if it is not contained in any open segment $]r,q[$ with $r,q\in\partial E$,
whereas it is an exposed point of $E$ if there exists
a support hyperplane $H$ to $E$ with $H\cap E = \{p\}$.
We denote by $\extr E$ and $\expo E$, respectively,
the sets of extremal and exposed points of $E$, see Figure \ref{expoE}.
\begin{figure}[h]
	\centering
	\includegraphics[width=13cm]{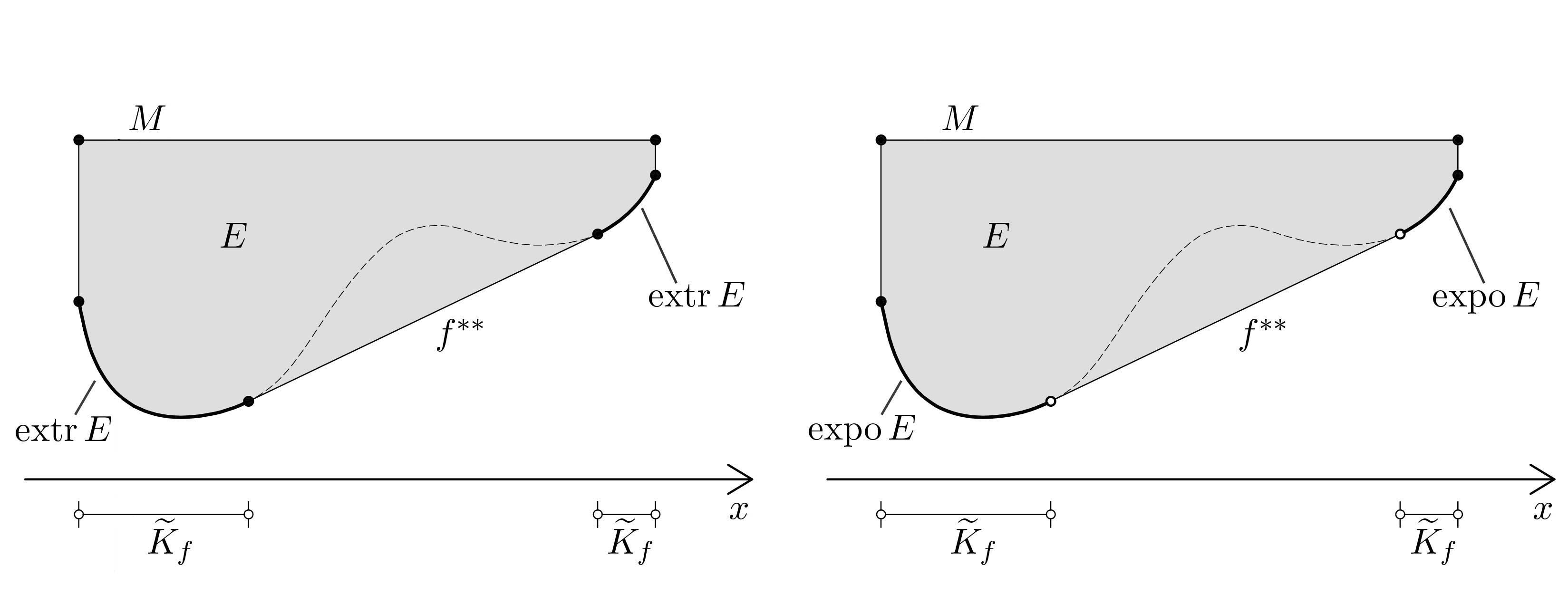}
	\caption{An example of $f^{**}(x)$ in which $\extr E$ and $\expo E$ are in evidence.}
	\label{expoE}
\end{figure}


Let us prove the following preliminary
\begin{lemma}\label{l:stab}
	Let $f\in C^1(\overline{\mathcal{I}})$,
	and let $(f_n)\subset  C^1(\overline{\mathcal{I}})$
	be a sequence converging uniformly to $f$.
	Then it holds:
	\begin{itemize}
		\item[(a)]
		If $x_0 \in \widetilde{K}_f$ and, for every $n\in\mathbb{N}$,
		$[a_n, b_n]\subset \overline{\mathcal{I}}$, $\lambda_n\in [0,1]$
		satisfy
		\[
		x_0 = (1-\lambda_n) a_n + \lambda_n b_n,
		\qquad
		f^{**}_n(x_0) = (1-\lambda_n) f_n(a_n) + \lambda_n f_n(b_n),
		\]
		then $a_n, b_n \to x_0$.
		
		\item[(b)]
		If, in addition, $f$ satisfies \eqref{noflat},
		i.e.\ $\overline{K_f} = \overline{\widetilde{K}_f}$,
		and $(a_0, b_0)$ is one of the maximal intervals $I^i$ where
		$f^{**}$ is affine,
		then for every $n\in\mathbb{N}$ there exists a maximal interval
		$(a_n, b_n)$ where $f_n^{**}$ is affine
		such that $a_n \to a_0$, $b_n \to b_0$.
	\end{itemize}
\end{lemma}
\begin{proof}
	(a)
	Since $x_0$ is an exposed point of the epigraph of $f^{**}$,
	it holds
	\[
	f(x) \geq f^{**}(x) > f(x_0) + f'(x_0) (x-x_0)
	\qquad
	\forall x\neq x_0.
	\]
	Assume by contradiction that at least one of the sequences
	$(a_n)$, $(b_n)$ does not converge to $x_0$.
	Then there exists a subsequence $(n_j)$ such that
	$a_{n_j} \to \overline{a}$, $b_{n_j}\to \overline{b}$,
	with $\overline{a} \leq x_0 \leq \overline{b}$ and
	$\overline{a} < \overline{b}$.
	Moreover, we clearly have
	$\lambda_{n_j} \to \lambda := (x_0 - \overline{a}) / (\overline{b} - \overline{a})$.
	Hence,
	\small
	\[
	f^{**}(x_0) = \lim_j f_{n_j}^{**}(x_0)
	= \lim_j (1-\lambda_{n_j}) f_{n_j}(a_{n_j}) + \lambda_{n_j} f_{n_j}(b_{n_j})
	= (1-\lambda) f(\overline{a}) + \lambda f(\overline{b})
	> f^{**}(x_0),
	\]
	\normalsize
	a contradiction.
		
	(b) In view of (a), it is enough to prove that,	if $x_0 \in (a_0, b_0)$, 
	then $x_0\not\in K_{f_n}$ for $n$ large enough. Assume by contradiction that there exists a subsequence $(n_j)$
	such that $x_0\in K_{f_{n_j}}$ for every $j$, i.e.,
	$f_{n_j}(x_0) = f_{n_j}^{**}(x_0)$ for every $j$.
	Since, by assumption, $f > f^{**}$ on $(a_0, b_0)$, one has
	\[
	f(x_0) > f^{**}(x_0) = \lim_j f_{n_j}^{**}(x_0)
	= \lim_j f_{n_j}(x_0) = f(x_0),
	\]
	a contradiction.
\end{proof}

\begin{proof}[Proof of Proposition~\ref{p:prop variaz}]
	Let $M > \max_{[a,b]} f$, so that
	$E := \epi f^{**} \cap \{y \leq M\}$ is a compact convex
	subset of $\mathbb{R}^2$.
	Moreover, we have that
	\[
	\{(x, f(x)):\ x\in \widetilde{K}_f \cup \{a,b\}\}
	= \expo E \cap \{y < M\},
	\]
	i.e., the set at left-hand side coincides with the set of
	exposed points of $\epi f^{**}$. Let
	\[
	f_s := f + s\, \varphi,
	\quad
	f_s^{**} := (f_s)^{**},
	\qquad s\in\mathbb{R}.
	\]
	By the Dominated Convergence Theorem, Proposition~\ref{p:prop variaz}
	will be a consequence of the
	following pointwise convergences:
	\begin{align}
	& \lim_{s\to 0} \frac{f_s^{**}(x_0) - f^{**}(x_0)}{s} =
	\varphi(x_0)
	&{\rm if}\hspace{3mm} x_0 \in \widetilde{K}_f,\label{f:conv1}\\
	& \lim_{s\to 0^\pm} \frac{f_s^{**}(x_0) - f^{**}(x_0)}{s} =
	\pm(\varphi_i^\pm)^{**}(x_0)
	&{\rm if}\hspace{3mm} x_0 \in K^i,\ i\in J_C.\label{f:conv2}
	\end{align}
	
	\textit{Proof of \eqref{f:conv1}.} We have already observed that, if $x_0\in \widetilde{K}_f$,
	then $f^{**}(x_0) = f(x_0)$ and	$(x_0, f(x_0)) \in \expo E$. Since $f\in C^1$, by definition of exposed point we have that
	\begin{equation*}\label{f:expo}
	f(x) \geq f^{**}(x) > f(x_0) + f'(x_0)\, (x-x_0) =: h(x),
	\qquad
	\forall x\in [a,b],\ x\neq x_0.
	\end{equation*}
	For every $s\in\mathbb{R}$ let
	$a_s \in [a, x_0]$,
	$b_s \in [x_0, b]$,
	and $\lambda_s \in [0,1]$ be such that
	\begin{equation}\label{f:conv}
	x_0 = (1-\lambda_s) a_s + \lambda_s \, b_s,
	\qquad
	f_s^{**}(x_0) = (1-\lambda_s) f_s(a_s) + \lambda_s\, f_s(b_s)\,.
	\end{equation}
	
	Let us first prove that
	\begin{equation}\label{f:conv1p}
	\lim_{s\to 0^+} \frac{f_s^{**}(x_0) - f^{**}(x_0)}{s} =
	\varphi(x_0)\,.
	\end{equation}
	Since
	\[
	\frac{f_s^{**}(x_0) - f^{**}(x_0)}{s}
	\leq
	\frac{f_s(x_0) - f(x_0)}{s} = \varphi(x_0),
	\qquad \forall s> 0,
	\]
	it follows that
	\[
	\limsup_{s\to 0^+} \frac{f_s^{**}(x_0) - f^{**}(x_0)}{s}
	\leq \varphi(x_0),
	\]
	hence it remains to prove that
	\[
	l := \liminf_{s\to 0^+} \frac{f_s^{**}(x_0) - f^{**}(x_0)}{s}
	\geq \varphi(x_0).
	\]
	Let $s_n\searrow 0$ be a sequence such that
	\(
	l
	= \lim\limits_{n \to +\infty} \frac{f_{s_n}^{**}(x_0) - f^{**}(x_0)}{s_n}\,.
	\)
	Using \eqref{f:conv} it holds
	\begin{equation}\label{f:maindis}
	\begin{split}
	& \frac{f_{s_n}^{**}(x_0) - f^{**}(x_0)}{{s_n}}
	=
	\frac{(1-\lambda_{s_n}) f_{s_n}(a_{s_n}) + \lambda_{s_n}\, f_{s_n}(b_{s_n}) - f(x_0)}{{s_n}}
	\\ & =
	\frac{(1-\lambda_{s_n}) f(a_{s_n}) + \lambda_{s_n} \, f(b_{s_n}) - f(x_0)}{{s_n}}
	+ (1-\lambda_{s_n}) \varphi(a_{s_n}) + \lambda_{s_n} \, \varphi(b_{s_n})\,.
	\end{split}
	\end{equation}
	Since the sequence $\{f_{s_n}\}$ converges uniformly to $f$,
	by Lemma~\ref{l:stab}(i) it follows that
	$a_{s_n}, b_{s_n} \to x_0$,	hence
	the right-hand side of \eqref{f:maindis} converges to a quantity greater than or equal to
	$\varphi(x_0)$,
	so that $l \geq \varphi(x_0)$ and \eqref{f:conv1p} follows.
	The computation of the limit \eqref{f:conv1p} for $s\to 0^-$ can be done similarly,
	observing that the same inequalities as above hold with	reversed signs. Hence, we conclude that \eqref{f:conv1} holds.
	
	\textit{Proof of \eqref{f:conv2}.} We shall prove \eqref{f:conv2} only for $s\to 0^+$,
	being the proof for $s\to 0^-$ entirely similar.
	Let $i\in J_C$ and let us denote
	\[
	B := K^i \cap (K_f \cup \{a,b\}),
	\qquad
	A := K^i \setminus B.
	\]
	Clearly, the set $B$ is closed and contains both the endpoints of the interval $K^i$.
	
	It is not restrictive to assume that
	$f^{**}(x) = 0$ for every $x\in K^i$, so that
	\begin{equation}\label{f:assump}
	f(x) = 0
	\quad \forall x\in B,
	\qquad
	f(x) > 0
	\quad \forall x\in A.
	\end{equation}
	Let us extend the function $\varphi_i^+$ to $+\infty$ outside $K^i$.
	Since $\varphi \leq \varphi_i^+$ and
	$(f + s\, \varphi_i^+) (x) = s\, \varphi_i^+(x)$
	for every $x\in K^i$, we have that
	\[
	f_s^{**}(x) = (f+s\, \varphi)^{**}(x) \leq
	(f +s\, \varphi_i^+)^{**}(x) = s\, (\varphi_i^+)^{**}(x),
	\qquad \forall x\in K^i,
	\]
	hence
	\[
	\limsup_{s\to 0^+}
	\frac{f_s^{**}(x_0) - f^{**}(x_0)}{s} \leq
	(\varphi_i^+)^{**}(x_0),
	\qquad \forall x_0 \in K^i.
	\]
	
	Let $s_n \searrow 0$ be a sequence such that
	\[
	l := \liminf_{s\to 0^+}
	\frac{f_s^{**}(x_0) - f^{**}(x_0)}{s}
	= \lim_{n\to +\infty}
	\frac{f_{s_n}^{**}(x_0) - f^{**}(x_0)}{s_n}\,,
	\]
	and let $E_n := \epi f_{s_n}^{**} \cap \{ y\leq M\}$, $n\in\mathbb{N}$.
	By \eqref{f:assump},
	for every $\varepsilon > 0$,
	there exists $N_\varepsilon\in\mathbb{N}$ such that,
	for $n \geq N_\varepsilon$, the extreme points of $E_n$
	belong to $B + B_\varepsilon$, so that
	\begin{equation}\label{f:incl}
	K_{f_{s_n}} \cap K^i \subset B + B_\varepsilon
	\qquad \forall n \geq N_\varepsilon.
	\end{equation}
	
	Let
	\[
	\varphi_\epsilon(x) :=
	\begin{cases}
	\varphi(x) & x\in B + B_\varepsilon,\\
	+ \infty & \text{otherwise},
	\end{cases}
	\]
	so that $\varphi_\epsilon(x) = \varphi_i^+(x)$ for all $x\in B$ and $\varphi_\epsilon \to \varphi_i^+$ pointwise in $K_i$.
	From \eqref{f:incl} we know that $f_{s_n}^{**}(x) = (f + s_n \, \varphi_\epsilon)^{**}(x)$ for all $x\in K^i$ and all $n \geq N_\varepsilon$ so that
	\[
	l = \lim_{n\to +\infty}
	\frac{f_{s_n}^{**}(x_0) - f^{**}(x_0)}{s_n}
	\geq
	\liminf_{n\to +\infty}
	\frac{(f + s_n \, \varphi_\epsilon)^{**}(x_0)-f^{**}(x_0)}{s_n}
	\geq \varphi_\varepsilon^{**}(x_0).
	\]
	Finally, letting $\varepsilon \to 0$, we conclude that
	$l\geq (\varphi_i^+)^{**}(x_0)$, concluding the proof.
\end{proof}\par
\begin{proof}[Proof of Theorem~\ref{prop variaz gen}]
	Since $f$ satisfies assumption \eqref{noflat},
	we can use the same arguments of Step~1 in Proposition~\ref{p:prop variaz}.
	We omit the details.
\end{proof}\par
\begin{proof}[Proof of Proposition~\ref{l:base}]
	If $f\in C^0([a,b])$, then $Tf \in C^0([a,b])$, $F\in C^1([a,b])$, $\Fss\in C^1([a,b])$, and
	$F(a) = \Fss(a) = 0$, $F(b) = \Fss(b)$. Hence,
	\begin{equation}\label{f:p1}
	\int_a^b Tf(y) \, dy = \int_a^b f(y)\, dy
	\qquad
	\forall f\in C^0([a,b]).
	\end{equation}
	In the following, we shall denote by $K_F$ the contact set of $F$,
	defined by
	$
	K_F := \{x\in (a,b):\ F(x) = \Fss(x)\}.
	$
	We remark that
	$f = Tf$ on $K_F$.
	Moreover, we have the following characterization of $K_F$:
	\begin{equation}\label{f:Kf}
	\begin{split}
	K_F & = \left\{ x\in (a,b):\ F(y) - F(x) - (y-x) F'(x) \geq 0,
	\ \forall y\in [a,b]\right\}
	\\ & =
	\left\{ x\in (a,b):\ \int_x^y [f(s) - f(x)]\, ds \geq 0,
	\ \forall y\in [a,b]\right\}\,.
	\end{split}
	\end{equation}
	
	Let $f, g\in C^0([a,b])$ and let $F(x) := \int_a^x f(y)\, dy$,
	$G(x) := \int_a^x g(y)\, dy$, $x\in [a,b]$.	We claim that
	\begin{equation}\label{f:p2}
	f,g\in C^0([a,b]),
	\ f \leq g
	\qquad\Longrightarrow\qquad
	Tf \leq Tg.
	\end{equation}
	
	Before proving \eqref{f:p2}, we observe that, from \eqref{f:p1}, \eqref{f:p2} and \cite[Proposition 1]{CranTar},
	we can conclude that \eqref{f:th} holds. So, it remains to prove \eqref{f:p2}. It will be convenient to perform a couple of reductions.
	
	\textit{Reduction 1:}
	it is not restrictive to prove \eqref{f:p2} under the additional assumption
	\begin{equation}
	\label{f:red1}
	f(a) = Tf(a),
	\quad
	f(b) = Tf(b),
	\qquad
	g(a) = Tg(a),
	\quad
	g(b) = Tg(b).
	\end{equation}
	Specifically, since
	\[
	Tf(a) = \inf_{x\in (a,b]} \frac{F(x) - F(a)}{x-a}
	= \inf_{x\in (a,b]} \frac{1}{x-a}\int_a^x f(s)\, ds,
	\]
	it is clear that $Tf(a) \leq f(a)$.
	If $Tf(a) < f(a)$, then there exists $\beta\in (a,b]$
	such that $(a, \beta)$
	is a connected component of $(a,b)\setminus K_F$.
	Given $\varepsilon > 0$,
	let $c_\varepsilon :=  (1 + 1/\sqrt{2}) \varepsilon$,
	and define the function
	\[
	\varphi_\varepsilon(t) :=
	\begin{cases}
	-1 + t/\varepsilon & \text{if}\
	t \in [0, c_\varepsilon],\\
	-1 + (2c_\varepsilon-t)/\varepsilon & \text{if}\
	t \in [c_\varepsilon, 2c_\varepsilon- \varepsilon],\\
	0 &\text{otherwise},
	\end{cases}
	\]
	so that $\varphi_\varepsilon(0) = -1$
	and $\int_0^{2c_\varepsilon - \varepsilon} \varphi_\varepsilon = 0$.	
	It is not difficult to show that,
	for $\varepsilon > 0$ small enough,
	the function
	$f_\varepsilon(x) := f(x) + [f(a) - Tf(a)]\, \varphi_\varepsilon(x-a)$
	satisfies $Tf_\varepsilon = Tf$ and $Tf_\varepsilon(a) = f_\varepsilon(a)$.
	Moreover, we have $\|f_\varepsilon - f\|_1 \to 0$ as
	$\varepsilon \to 0$; similarly, we can modify $f$ near $b$ and the same
	for $g$.

	\textit{Reduction 2:}
	it is not restrictive to prove \eqref{f:p2} under the additional assumption
	\begin{equation}
	\label{f:red2}
	f < g \qquad \text{in}\ [a,b].
	\end{equation}
	Specifically, it is enough to check that
	$T(g + \varepsilon) = Tg + \varepsilon$.

We are thus reduced to prove \eqref{f:p2} when $f,g \in C([a,b])$ satisfy also \eqref{f:red1} and \eqref{f:red2}.
	Let
	\[
	x_0 := \max\{x\in [a,b]:\
	Tf(y) \leq Tg(y)\
	\forall y\in [a, x]\}.
	\]
	Since $F\leq G$, we have that $\Fss \leq \Gss$ and hence $Tf(a) \leq Tg(a)$.
	Moreover, by \eqref{f:red1} and \eqref{f:red2}, we clearly have $x_0 > a$.

	Assume by contradiction that $x_0 < b$,
	so that $Tf(x_0) = Tg(x_0)$, and let us consider the following cases.
	
	\textit{Case 1:} $x_0 \in K_F\cap K_G$.
	Hence,
	\[
	Tf(x_0) = f(x_0) < g(x_0) = Tg(x_0),
	\]
	in contradiction with $Tf(x_0) = Tg(x_0)$.
	
	\textit{Case 2:} $x_0 \in K_F$, $x_0\not\in K_G$.
	Let $(\alpha, \beta)$ be the maximal connected component of $(a,b)\setminus K_G$
	containing $x_0$, so that $Tg$ is constant on $[\alpha, \beta]$
	and $g(\alpha) = Tg(\alpha)$, $g(\beta) = Tg(\beta)$.
	Here it is worth to remark that these equalities hold also in the
	case $\alpha = a$ or $\beta = b$ thanks to \eqref{f:red1}. By the characterization \eqref{f:Kf} we have that:
	
	\small
	\begin{gather*}
	x_0 \in K_F
		\hspace{1mm}\Longrightarrow	\hspace{1mm}
	\int_{x_0}^{\beta} [f(s) - f(x_0)]\, ds \geq 0,\qquad
	\beta\in K_G\ \text{or}\ \beta = b
	\hspace{1mm}\Longrightarrow\hspace{1mm}
	\int_{x_0}^{\beta} [g(\beta) - g(s)]\, ds \geq 0,
	\end{gather*}
	\normalsize
	so that
	\begin{equation}\label{f:dis1}
	\int_{x_0}^{\beta}[g(\beta) - g(s) + f(s) - f(x_0)]\, ds \geq 0.
	\end{equation}
	On the other hand
	$g(\beta) = Tg(\beta) = Tg(x_0)$
	and $f(x_0) = Tf(x_0)$.
	Since $Tf(x_0) = Tg(x_0)$, we conclude that $g(\beta) = f(x_0)$,
	hence from \eqref{f:dis1} it holds
	\[
	\int_{x_0}^{\beta}[f(s) - g(s) ]\, ds \geq 0,
	\]
	contradicting the assumption $f < g$.
	
	\textit{Case 3:} $x_0 \not\in K_F$, $x_0\in K_G$.
	We can reason as in the previous case,
	considering the connected component $(\alpha, \beta)$ of $(a,b)\setminus K_f$
	containing $x_0$, and obtaining the inequality
	\begin{equation*}\label{f:dis2}
	\int_{\alpha}^{x_0}[g(x_0) - g(s) + f(s) - f(\alpha)]\, ds \geq 0.
	\end{equation*}
	Since, here, $g(x_0) = f(\alpha)$, we get again
	a contradiction with the assumption $f<g$.
	
	\textit{Case 4:} $x_0 \not\in K_F$, $x_0\not\in K_G$.
	In this case $Tf$ and $Tg$ are locally constant in a neighborhood of $x_0$,
	again contradicting the definition of $x_0$. This proves \eqref{f:th}.\par

	We divide the proof of \eqref{f:th2} into two steps.\par\noindent
	\textit{STEP 1.}
	If $f\in L^1(a,b)$ and the sequence $\{f_n\} \subset C^0([a,b])$ converges to $f$ in $L^1$, then
	\[
	Tf_n \to Tf\quad\text{a.e.},
	\quad \text{and} \quad
	\|Tf_n - Tf\|_1 \to 0.
	\]
	Let $F(x) := \int_a^x f$, $F_n(x) := \int_a^x f_n$.
	We have that $F_n \to F$ uniformly in $[a,b]$,
	hence also $\Fss_n \to \Fss$ uniformly in $[a,b]$.
	(Proof: use the characterization
	$\Fss(x) = \min\{(1-\lambda) F(x_0) + \lambda F(x_1):\
	\lambda\in [0,1],\ (1-\lambda)x_0 + \lambda x_1 = x\}$.)
	By Theorem 24.5 in \cite{Rock} we deduce that
	$(\Fss_n)' \to (\Fss)'$ at every point of differentiability
	of $\Fss$, i.e.\ almost everywhere in $[a,b]$.
	By definition of $T$, it follows that
	$Tf_n \to Tf$ a.e. in $[a,b]$.
	
	The $L^1$ convergence of $\{Tf_n\}$ to $Tf$ follows from \eqref{f:th}.
	Namely, from \eqref{f:th} we have that $\|Tf_n - Tf_m\|_1 \leq \|f_n - f_m\|_1$,
	hence $\{Tf_n\}$ is a Cauchy sequence in $L^1$ (and so it converges to its pointwise limit), thereby proving the claim of Step 1.\par\noindent	
	\textit{STEP 2.} Given $f,g\in L^1(a,b)$, let $\{f_n\}, \{g_n\} \subset C^0([a,b])$ sequences
	converging in $L^1$ respectively to $f$ and $g$.
	By \eqref{f:th}, we have that $\|Tf_n - Tg_n\|_1 \leq \|f_n - g_n\|_1$ for every $n$.
	Hence, by Step~1, passing to the limit as $n\to +\infty$
	we obtain \eqref{f:th2}.
\end{proof}

\begin{proof}[Proof of Proposition~\ref{lipschitz Iphi}]
	By 
	(\ref{zbar}) we have that
	\footnotesize
	\begin{equation*}
	\label{zbar3}
	Tf=
	\begin{cases}
	\dfrac{F(b^i)-F(a^i)}{b^i-a^i} & x\in I^i,\ i \in J,\\
	F'(x) & x \in {K}_F.
	\end{cases}\qquad
	(\mathcal{J}^\varphi_F)'=
	\begin{cases}
	\dfrac{\varphi(b^i)-\varphi(a^i)}{b^i-a^i}=\dfrac{\int_{a^i}^{b^i}\varphi'dx}{b^i-a^i} & x\in I^i,
	\ i \in J,\\
	\varphi'(x) & x \in {K}_F,
	\end{cases}
	\end{equation*}
	\normalsize
	and similarly for $Tg$ and $(\mathcal{J}^\varphi_G)'$. Then we have the fundamental integral equivalence
	\begin{equation*}
	\int_a^b [Tf\hspace{1mm} (\mathcal{J}^\varphi_F)'-Tg\hspace{1mm} (\mathcal{J}^\varphi_G)']dx=\int_a^b [Tf-Tg]\hspace{1mm} \varphi'dx,\hspace{5mm}\forall\varphi\in C^{\infty}_c(\mathcal{I})
	\label{crucial}
	\end{equation*}
	so that the thesis is achieved by applying the H\"{o}lder inequality and \eqref{f:th2}.
\end{proof}
\begin{proof}[Proof of Proposition~\ref{rmk}]
	i) Let us observe that
	\small
	\begin{equation*}
	\int_a^b [\mathcal{H}(Tf)\hspace{1mm} (\mathcal{J}^\varphi_F)'-\mathcal{H}(Tg)\hspace{1mm} (\mathcal{J}^\varphi_G)']dx=\int_a^b [\mathcal{H}(Tf)-\mathcal{H}(Tg)]\hspace{1mm} \varphi'dx,\hspace{5mm}\forall\varphi\in C^{\infty}_c(\mathcal{I})
	\label{crucial2}
	\end{equation*}
	\normalsize
	then we proceed as in Proposition \ref{lipschitz Iphi}, considering the Lipschitz property of $\mathcal{H}$.\par
	ii)	We consider the equivalence

\footnotesize
	\begin{equation}
	\begin{split}
	&\int_a^b [\mathcal{H}(Tf)\hspace{1mm} (\mathcal{G}^{F,\psi})'-\mathcal{H}(Tg)\hspace{1mm} (\mathcal{G}^{G,\psi})']dx=\int_a^b [\mathcal{H}(Tf)(\psi\cos F)'-\mathcal{H}(Tg)(\psi\cos G)']dx\\&=\int_a^b [\mathcal{H}(Tf)-\mathcal{H}(Tg)](\psi\cos F)'dx+\int_a^b\mathcal{H}(Tg)\hspace{1mm}(\psi\cos F-\psi\cos G)'dx \hspace{5mm}\forall\psi\in C^{\infty}_c(\mathcal{I}).
	\end{split}
	\label{crucial3}
	\end{equation}
	\normalsize
	Since $(\psi\cos F-\psi\cos G)'=\psi'(\cos F-\cos G)-\psi[(f-g)\sin F+g(\sin F-\sin G)]$, the thesis follows by the Lipschitzianity of $\sin,\cos$
and by arguing as in the above proofs.
\end{proof}

\begin{proof}[Proof of Proposition~\ref{Iphi conv}]
	The functions $\mathcal{J}^\varphi_n$ are Lipschitz continuous, with
	\[
	\|\mathcal{J}^\varphi_n\|_\infty \leq \|\varphi\|_\infty,
	\qquad
	\|(\mathcal{J}^\varphi_n)'\|_\infty \leq \|\varphi'\|_\infty.
	\]
	Hence, by the Dominated Convergence Theorem,
	it is enough to show that $\mathcal{J}^\varphi_n \to \mathcal{J}^\varphi$ and $(\mathcal{J}^\varphi_n)' \to (\mathcal{J}^\varphi)'$
	a.e.\ in $\mathcal{I}$.	
	The pointwise convergence of $\{\mathcal{J}^\varphi_n\}$ to $\mathcal{J}^\varphi$ is a
	direct consequence of Lemma~\ref{l:stab}. To prove the a.e.\ convergence of $\{(\mathcal{J}^\varphi_n)'\}$ to $(\mathcal{J}^\varphi)'$,
	it will be convenient to distinguish between the two cases (a) and (b)
	in Lemma~\ref{l:stab}.
	
	Let $x_0$ be as in case (a), and assume that all the function $\mathcal{J}^\varphi_n$
	are differentiable at $x_0$ (this condition is satisfied at
	almost every point).
	For every $n$,
	if $x_0 \in K_{f_n}$ (i.e.\ $a_n = b_n$) then $(\mathcal{J}^\varphi_n)'(x_0) = \varphi'(x_0)$,
	otherwise there exists $x_n\in (a_n, b_n)$
	such that
	\[
	(\mathcal{J}^\varphi_n)'(x_0) = \frac{\varphi(b_n) - \varphi(a_n)}{b_n - a_n}
	= \varphi'(x_n).
	\]
	Since $a_n, b_n\to x_0$, we finally get $(\mathcal{J}^\varphi_n)'(x_0) \to \varphi'(x_0) = (\mathcal{J}^\varphi)'(x_0)$.
	
	Let $x_0$ be as in case (b).
	Then, for $n$ large enough,
	\[
	(\mathcal{J}^\varphi_n)'(x_0) = \frac{\varphi(b_n) - \varphi(a_n)}{b_n - a_n}
	\to
	\frac{\varphi(b_0) - \varphi(a_0)}{b_0 - a_0}
	= (\mathcal{J}^{\varphi})'(x_0),
	\]
	and the proof is complete.
\end{proof}
\section{Proofs of existence and uniqueness results}
\label{Proof of the Theorem 2}
To simplify the notations we define the functionals
\small
\begin{equation}
\chi(u):=\dfrac{u'}{\sqrt{1+(u')^2}},\hspace{10mm}\gamma(u):=\sqrt{1+(u')^2}
\label{chi, gamma}
\end{equation}
\normalsize
and we state a preliminary
\begin{lemma}
	 $\chi, \gamma: C^1[0,L]\rightarrow C^0[0,L]$ are locally Lipschitz continuous.
	\label{lipschitz nonlin}
\end{lemma}
\begin{proof}[Proof]
	Given $v,w \in C^1[0,L]$, we apply the Lagrange Theorem, so that there exists $\varrho:= \varrho(x)\in\big(v',w'\big)$ such that

	\small
	\begin{equation*}
	\begin{split}
	&\big|\chi\big(v\big)-\chi\big(w)\big)\big|=\frac{\big|v'-w'\big|}{\sqrt{(1+\varrho^2)^3}}\leq \big|v'-w'\big|,\hspace{5mm}\big|\gamma\big(v\big)-\gamma\big(w)\big)\big|=\frac{|\varrho|\big|v'-w'\big|}{\sqrt{1+\varrho^2}}\leq \big|v'-w'\big|.
	\end{split}
	\label{lagr}
	\end{equation*}
	\normalsize
\end{proof}

\begin{proof}[Proof of Theorem~\ref{teo1}]
	Let $n\geq 1$ an integer.
	Testing $n$ times equations (\ref{eq weak sist gal}) for $r=1,\dots,n$ and $t\geq 0$ we obtain a system of ODE's for $k=1,\dots,n$

	\footnotesize
	\begin{equation}
	\begin{cases}
	M\ddot{w}^k_n(t)+EI\dfrac{k^4\pi^4}{L^4} w^k_n(t)+Mg\dfrac{\sqrt{2L}((-1)^k-1)}{k\pi}=\\\hspace{15mm}\big(h_\alpha(w_n,\theta_n),[\mathcal{J}^{e_k}_\alpha]'\big)_2+\big(h_\beta(w_n,\theta_n),[\mathcal{J}^{e_k}_\beta]'\big)_2\vspace{3mm}\\
	 \dfrac{M\ell}{3}\ddot{\theta}^k_n(t)+\bigg(EJ\dfrac{k^4\pi^4}{L^4\ell}+GK\dfrac{k^2\pi^2}{L^2\ell}\bigg)\theta^k_n(t)=\\\hspace{15mm}\big(h_\alpha(w_n,\theta_n),[\mathcal{G}^{\theta_n,e_k}_\alpha]_x\big)_2-\big(h_\beta(w_n,\theta_n),[\mathcal{G}^{\theta_n,e_k}_\beta]_x\big)_2
	\end{cases}
	\label{eq weak sist lin gal  2n}
	\end{equation}
	\normalsize
	with the initial conditions as in \eqref{eq weak sist gal}. The local existence of a solution $(w^k_n,\theta^k_n)$ for all $k=1,\dots,n$ and $t\geq 0$ depends on the regularity of the right hand side terms of (\ref{eq weak sist lin gal  2n}). We introduce the vectors $W=[w^1_n,\dots,w^n_n]$, $\Theta=[\theta^1_n,\dots,\theta^n_n]$ and $e(x)=[e_1(x),\dots,e_n(x)]$ in $\mathbb{R}^n$; we study the nonlinearities related to one cable.\par
If the functions

	\footnotesize
	\begin{equation*}
	\begin{split}
	&F_k(W,\Theta):=\big[H\overline{\xi}+\frac{AE_c}{L_{c}}\Gamma\big(W\cdot e+\ell\sin(\Theta\cdot e)\big]\int_{0}^{L}\chi\bigg([W\cdot e+\ell\sin(\Theta\cdot e)+y]^{**}\bigg)\hspace{1mm}(\mathcal{J}^{e_k}_\alpha)'dx,\\&G_k(W,\Theta):=\big[H\overline{\xi}+\frac{AE_c}{L_{c}}\Gamma\big(W\cdot e+\ell\sin(\Theta\cdot e)\big]\int_{0}^{L}\chi\bigg([W\cdot e+\ell\sin(\Theta\cdot e)+y]^{**}\bigg)\hspace{1mm}(\mathcal{G}^{\Theta\cdot e,e_k}_\alpha)'dx,
	\end{split}
	\end{equation*}
	\normalsize
	are locally Lipschitz continuous with respect to $W$, $\Theta$ for all $k=1,\dots,n$, we have the existence and uniqueness of a solution of \eqref{eq weak sist lin gal  2n} on some interval  $[0,t_n)$ with $t_n\in(0,T]$. \par
	Thanks to Lemma \ref{lipschitz nonlin}, Proposition \ref{rmk}-i) and (\ref{gamma}), for every compact subset $X\subset \mathbb{R}^n$ there exists $C_0>0$ such that, for every $W_1, W_2, \Theta_1, \Theta_2\in X$ we have

\footnotesize
	\begin{equation*}
	\begin{split}
	&\big|F_k(W_1,\Theta_1)-F_k(W_2,\Theta_2)\big|=\bigg|\big[H\overline{\xi}+\frac{AE_c}{L_{c}}\Gamma\big(W_1\cdot e+\ell\sin(\Theta_1\cdot e)\big)\big]\cdot\\&\cdot\int_{0}^{L}\bigg[\chi\bigg([W_1\cdot e+\ell\sin(\Theta_1\cdot e)+y]^{**}\bigg)[(\mathcal{J}^{e_k}_{\alpha })_1]'-\chi\bigg([W_2\cdot e+\ell\sin(\Theta_2\cdot e)+y]^{**}\bigg)[(\mathcal{J}^{e_k}_\alpha)_2]'\bigg]dx\\&+\frac{AE_c}{L_{c}}\bigg\{\int_{0}^{L}\bigg[\gamma\bigg([W_1\cdot e+\ell\sin(\Theta_1\cdot e)+y]^{**}\bigg)-\gamma\bigg([W_2\cdot e+\ell\sin(\Theta_2\cdot e)+y]^{**}\bigg)\bigg]dx\bigg\}\cdot\\&\bigg\{\int_{0}^{L}\chi\bigg([W_2\cdot e+\ell\sin(\Theta_2\cdot e)+y]^{**}\bigg)[(\mathcal{J}^{e_k}_\alpha)_2]'dx\bigg\}\bigg|\leq C_0\|e'_k\|_{\infty}\big(|W_1-W_2|+|\Theta_1-\Theta_2|\big)\|e\|_{W^{1,1}},
	\end{split}
	\label{F continua}
	\end{equation*}
	\normalsize
	so that $F_k(W,\Theta)$ is locally Lipschitz continuous for all $k=1,\dots,n$.
	With some additional computations due to the presence of the trigonometric functions, see Proposition \ref{rmk}-ii), the same arguments can be applied to obtain the locally Lipschitz continuity of $G_k(W,\Theta)$.\par
	We now seek a uniform bound for the sequence $(w_n,\theta_n)$. We omit for the moment the spatial dependence of the approximated solutions. We test the first equation in (\ref{eq weak sist gal}) by $\dot{w}_n$, the second by $\dot{\theta}_n$ and we sum the two equations. Hence, we obtain
	
\footnotesize
	\begin{equation}
	\begin{split}
	 &\frac{d}{dt}\bigg[\dfrac{M}{2}\|\dot{w}_n\|^2_2+\dfrac{EI}{2}\|w_n\|^2_{H^2}+\dfrac{M\ell^2}{6}\|\dot{\theta}_n\|^2_2+\dfrac{EJ}{2}\|\theta_n\|^2_{H^2}+\dfrac{GK}{2}\|\theta_n\|^2_{H^1}-Mg\int_{0}^{L}w_ndx\bigg]=\\&\int_{0}^Lh_\alpha(w_n,\theta_n)[\mathcal{J}^{\dot{w}_n}_\alpha+\ell\mathcal{G}^{\theta_n,\dot{\theta}_n}_\alpha]_{x}dx+\int_{0}^Lh_\beta(w_n,\theta_n)[\mathcal{J}^{\dot{w}_n}_\beta-\ell\mathcal{G}^{\theta_n,\dot{\theta}_n}_\beta]_{x}dx.
	\end{split}
	\label{step 2}
	\end{equation}
	\normalsize
	Let us define the energy of the system for the approximate solution $(w_n, \theta_n)$ as

	\footnotesize
	\begin{equation*}
	\begin{split}
	 \mathcal{E}_n(t):=&\dfrac{M}{2}\|\dot{w}_n\|^2_2+\dfrac{EI}{2}\|w_n\|^2_{H^2}+\dfrac{M\ell^2}{6}\|\dot{\theta}_n\|^2_2+\dfrac{EJ}{2}\|\theta_n\|^2_{H^2}+\dfrac{GK}{2}\|\theta_n\|^2_{H^1}-Mg\int_{0}^{L}w_ndx\\&+H\overline{\xi}\int_{0}^{L}(\sqrt{1+\{[(w_n+\ell\sin\theta_n+y)^{**}]_x\}^2}+\sqrt{1+\{[(w_n-\ell\sin\theta_n+y)^{**}]_x\}^2})dx\\&+\dfrac{AE_c}{2L_c}\big([\Gamma(w_n+\ell\sin\theta_n)]^2+[\Gamma(w_n-\ell\sin\theta_n)]^2\big).
	\end{split}
	\end{equation*}
	\normalsize
	Since we are in the finite dimensional setting it holds the assumption \eqref{noflat}, so that we apply Corollary \ref{prop variaz} and Theorem \ref{prop variaz gen}, finding the energy conservation.
	This is the point where we take advantage of the final dimensional nature of the problem.
	Hence from \eqref{step 2} we have $\dot{\mathcal{E}}_n(t)=0$, that is

	\footnotesize
	\begin{equation}
	\begin{split}
	 &\mathcal{E}_n(t)=\mathcal{E}_n(0)=\frac{M}{2}\|w^1_n\|^2_2+\frac{EI}{2}\|w^0_n\|^2_{H^2}+\frac{M\ell^2}{6}\|\theta^1_n\|^2_2+\dfrac{EJ}{2}\|\theta^0_n\|^2_{H^2}+\frac{GK}{2}\|\theta^0_n\|^2_{H^1}\\&+H\overline{\xi}\int_{0}^{L}(\sqrt{1+\{[(w^0_n+\ell\sin\theta^0_n+y)^{**}]'\}^2}+\sqrt{1+\{[(w^0_n-\ell\sin\theta^0_n+y)^{**}]'\}^2})dx\\&+\dfrac{AE_c}{2L_c}\bigg(\int_{0}^{L}[\sqrt{1+\{[(w^0_n+\ell\sin\theta^0_n+y)^{**}]'\}^2}] dx-L_c\bigg)^2\\&+\dfrac{AE_c}{2L_c}\bigg(\int_{0}^{L}[\sqrt{1+\{[(w^0_n-\ell\sin\theta^0_n+y)^{**}]'\}^2}] dx-L_c\bigg)^2-Mg\int_{0}^{L}w^0_ndx.
	\label{step 2.2}
	\end{split}
	\end{equation}
	\normalsize
	We recall the Poincar\'e inequality $\|w\|_2\leq \varLambda \|w\|_{H^2}$ for every $w\in H^2\cap H^1_0$ ($\varLambda>0$) and we observe that in $\mathcal{E}_n(t)$ only the gravitational term has undefined sign. In order to estimate this term we notice that for all $\varepsilon\in(0,\frac{1}{4}]$ we have
		\footnotesize
	$$-\int_{0}^{L}w_ndx\geq -\int_{0}^{L}(1+\varepsilon w_n^2)dx=-(L+\varepsilon\|w_n\|_2^2)\geq-(L+\varepsilon\varLambda^2\|w_n\|_{H^2}^2).$$
		\normalsize
	Choosing a sufficiently small $\varepsilon\in(0,\frac{1}{4}]$, we find $\eta>0$ such that
	\footnotesize
	\begin{equation*}
	\begin{split}
	 &\mathcal{E}_n(t)\geq\frac{M}{2}\|\dot{w}_n\|_2^2+\bigg(\frac{EI}{2}-Mg\varLambda\varepsilon\bigg)\|w_n\|_{H^2}^2+\frac{M\ell^2}{6}\|\dot{\theta}_n\|^2_2+\frac{EJ}{2}\|\theta_n\|_{H^2}^2+\frac{GK}{2}\|\theta_n\|_{H^1}^2\\&\hspace{10mm}+H\overline{\xi}\int_{0}^{L}(\sqrt{1+\{[(w_n+\ell\sin\theta_n+y)^{**}]_x\}^2}+\sqrt{1+\{[(w_n-\ell\sin\theta_n+y)^{**}]_x\}^2})dx\\&\hspace{10mm}+\frac{AE}{2L_{c}}\big([\Gamma(w_n+\ell\sin\theta_n)]^2+[\Gamma(w_n-\ell\sin\theta_n)]^2\big)-MgL\\&\hspace{10mm}\geq\eta(\|\dot{w}_n\|^2_2+\|w_n\|^2_{H^2}+\|\dot{\theta}_n\|^2_2+\|\theta_n\|^2_{H^2}+\|\theta_n\|^2_{H^1})-MgL.
	\end{split}
	\label{step 2.6}
	\end{equation*}
	\normalsize
	Then from (\ref{step 2.2}) we infer for all $t\in [0,t_n)$ and $n\geq 1$
	\footnotesize
	\begin{equation}
	\begin{split}
	\eta(\|\dot{w}_n\|^2_2+\|w_n\|^2_{H^2}+\|\dot{\theta}_n\|^2_2+\|\theta_n\|^2_{H^2}+\|\theta_n\|^2_{H^1})\leq \mathcal{E}_n(0)+MgL.
	\label{step2.8}
	\end{split}
	\end{equation}
	\normalsize
	Thanks to \eqref{step2.8} we obtain $t_n=T$, ensuring the global existence and uniqueness of the solution $(w_n,\theta_n)$ on $[0,T]$. Moreover, since the total energy of (\ref{eq weak sist gal}) is conserved in time, the solution cannot blow up in finite time and the global existence is obtained for an arbitrary $T>0$, including $T=\infty$.
\end{proof}
\begin{proof}[Proof of Theorem~\ref{corol inclusione}]
	To simplify the notation we denote by $L^p(V)$ the space $L^p((0,T);V(0,L))$ for $1\leq p\leq\infty$, by $Q=(0,T)\times(0,L)$ and by $C>0$ all the positive constants.
	We observe that $\sup_n|\mathcal{E}_n(0)|+MgL<\infty$ is independent of $n$ and  $t$, since $w^0_n$ and $\theta^0_n$ belong to $C^1[0,L]$. Then from \eqref{step2.8} we infer the boundedness of $\{w_n\}, \{\theta_n\}$ in $L^{\infty}(H^2)$ and of $\{\dot{w}_n\}, \{\dot{\theta}_n\} $ in $L^{\infty}(L^2)$, implying, up to a subsequence, the weak* convergence respectively to $w$, $\theta$ and to $\dot{w}$, $\dot{\theta}$ in the previous spaces.
	In particular from the boundedness of $\{w_n\}, \{\theta_n\}$ and $\{\dot{w}_n\}, \{\dot{\theta}_n\}$ we also have weak convergence respectively in $L^{2}(H^2)$ and $L^{2}(Q)$; then, due to the compact embedding $H^1(Q)\subset L^2(Q)$, we obtain the strong convergence $w_{n}  \rightarrow w$, $\theta_{n}\rightarrow\theta$ in $L^{2}(Q)$, from which $\sin\theta_{n}\rightarrow\sin\theta $ in $L^2(Q)$, since  $\|\sin\theta_n-\sin\theta\|_{L^2(Q)}\leq\|\theta_n-\theta\|_{L^2(Q)}\rightarrow0$ as $n\rightarrow\infty$ (similarly $\cos\theta_{n}\rightarrow\cos\theta$).\par About $\Gamma$ in (\ref{gamma}), thanks to Lemma \ref{lipschitz nonlin}, \eqref{f:th2}, H\"{o}lder and Poincar\'e inequalities, there exists $C>0$ such that
	\small
	\begin{equation*}
	\begin{split}
	|\Gamma(w_n\pm\ell\sin\theta_n)-\Gamma(w\pm\ell\sin\theta)|\leq(\|(w_n-w)_x\|_1+\|\theta_n-\theta\|_{W^{1,1}})\\\hspace{15mm}\leq C(\|w_n-w\|_{L^\infty(H^2)}+\|\theta_n-\theta\|_{L^\infty(H^2)})\rightarrow 0,
	\end{split}
	\end{equation*}
	\normalsize
	up to a subsequence, implying $\Gamma(w_n\pm\ell\sin\theta_n)\rightarrow\Gamma(w\pm\ell\sin\theta)$.\par
	We consider the functional $\chi$ defined in (\ref{chi, gamma}), and we note that $|\chi(u)|<1$ for all $u\in H^2(0,L)\subset C^1[0,L]$. Then, we have $ \chi^2\big([w_n\pm\ell\sin\theta_n+y]^{**}\big)< 1$ and
	\begin{equation*}
	 \|\chi\big([w_n\pm\ell\sin\theta_n+y]^{**}\big)\|^2_{L^2(Q)}=\int_{0}^{T}\int_{0}^{L}\dfrac{\{[(w_n\pm\ell\sin\theta_n+y)^{**}]_{x}\}^2}{1+\{[(w_n\pm\ell\sin\theta_n+y)^{**}]_{x}\}^2}dx dt< LT.
	\end{equation*}
	Hence $\chi\big([w_n\pm\ell\sin\theta_n+y]^{**}\big)$ converges weakly, up to a subsequence, to $\chi\big([w\pm\ell\sin\theta+y]^{**}\big)$ in $L^2(Q)$. Therefore we pass to the limit the first equation in (\ref{eq weak sist gal}), since $\|(\mathcal{J}^{e_k}_\alpha)'\|_{\infty}\leq\|e'_k\|_{\infty}$. To do the same for the second equation in (\ref{eq weak sist gal}) we use (\ref{crucial3}) and the bounds
	\begin{equation*}
	\begin{split}
	&\|\chi\big([w_n\pm\ell\sin\theta_n+y]^{**}\big)\cos\theta_n\|^2_{L^2(Q)}< LT,\\&\|\chi\big([w_n\pm\ell\sin\theta_n+y]^{**}\big)\theta_{nx}\sin\theta_n\|^2_{L^2(Q)}\leq T\|\theta_{n}\|^2_{L^\infty(H^1)},
	\end{split}
	\end{equation*}
	which imply the weak convergence of these terms in $L^2(Q)$, up to a subsequence.\par
	Next, for any $n\geq 1$ we put
	\small
	\begin{equation*}
	\begin{split}
	&w^0_n:=\sum_{k=1}^{n} (w^0,e_k)_2 \hspace{1mm}e_k =\frac{L^4}{\pi^4}\sum_{k=1}^{n}\dfrac{(w^0,e_k)_{H^2}}{k^4} \hspace{1mm}e_k, \\
	&\theta^0_n:=\sum_{k=1}^{n} (\theta^0,e_k)_2 \hspace{1mm}e_k =\sum_{k=1}^{n}\bigg(EJ\frac{k^4\pi^4}{L^4}+GK\frac{k^2\pi^2}{L^2}\bigg)^{-1}[EJ(\theta^0,e_k)_{H^2}+GK(\theta^0,e_k)_{H^1}]\hspace{1mm}e_k,\\& w^1_n:=\sum_{k=1}^{n} (w^1,e_k)_2 \hspace{1mm}e_k,  \hspace{8mm} \theta^1_n:=\sum_{k=1}^{n}(\theta^1,e_k)_2 \hspace{1mm}e_k,
	\end{split}
	\end{equation*}
	\normalsize
	so that $w^0_n\rightarrow w^0$, $\theta^0_n\rightarrow\theta^0$ in $H^2$ and
	$w^1_n\rightarrow w^1$, $\theta^1_n\rightarrow\theta^1$ in $L^2$
	as $n\rightarrow\infty$.\par
	Therefore we pass to the limit the problem \eqref{eq weak sist gal}, so that there exists an approximable solution of (\ref{eq weak sist2})-(\ref{ics}), $(w,\theta)\in X^2_T$ in the sense of Definition \ref{defapproximable}.
\end{proof}

\section{Conclusions}\label{concl}

In this paper, we proposed a model in which both the hangers and the cables are deformable. This was previously considered in
\cite{Gazzola hyperb} through a four DOF model, but here we have only focused on the vertical displacements and the torsional rotations of the deck,
thereby dealing with a two DOF model. In this case it appears out of reach to obtain a precise explanation of the origin
of torsional instability in terms of Poincar\'e maps as in \cite{argaz}. However, our numerical results still show the same qualitative phenomenon:
after exceeding a certain energy threshold the system becomes unstable and sudden and violent torsional oscillations appear.\par
The analyisis of this new model for suspension bridges requires the study of the variation of an energy functional depending on the convexification of the involved functions. The computation of the Gateaux derivatives of the functional is quite involved and, apart of \textquotedblleft spoiling''  the action of the smooth
test functions, it does not exist in some situations.
After a full energy balance, in Section \ref{The variational problem} we derived the weak form of the system of nonlinear nonlocal partial differential inclusions. This system is nonlinear due to the convexification, the geometric configuration of the cables and the rotation of the deck; moreover,
we avoided the linearization based on smallness assumptions on the torsional angle of the deck.\par
The typical behavior of civil structures enabled us to consider approximable solutions as representative enough of our problem, thereby reducing
to a system of ordinary differential equations, through the Galerkin procedure. We then proved existence of weak approximable solutions.
This enabled us to study the problem numerically, considering 10 longitudinal modes interacting with 4 torsional modes and we found a threshold
of torsional instability for each longitudinal mode excited. We compared these thresholds with those of the correspondent model
without convexification. Our numerical results show that, for structures displaying only low modes of vibration (e.g. 1$^{st}$,
2$^{nd}$, 3$^{rd}$), we may assume inextensible hangers, reducing the computational costs and obtaining safe instability thresholds. On the other hand,
if the structure vibrates on higher modes (e.g. 9$^{th}$, 10$^{th}$) as the TNB, this assumption may provide overestimated thresholds. Here the slackening
of the hangers increases dangerously the torsional instability; this fact should be a warning for the designers of bridge structures, that are able to
exhibit, in realistic situations, large vibration frequencies. Let us recall that the wind velocity determines the excited mode, see \cite[pp.21-27]{TNB}
and that an explicit rule has been recently found in \cite{bonedegaz}. It turns out that the longitudinal modes that were torsionally unstable at the TNB
were the 9$^{th}$ and 10$^{th}$, precisely the ones for which we found lower thresholds of instability in the new model with convexification (that is, with
hangers slackening).\par\medskip\noindent
{\bf Acknowledgements.} The third Author is partially supported by the PRIN project {\em Equazioni alle derivate parziali di tipo ellittico e parabolico:
aspetti geometrici, disuguaglianze collegate, e applicazioni} and by the Gruppo Nazionale per l'Analisi Matematica, la Probabilit\`a e le loro
Applicazioni (GNAMPA) of the Istituto Nazionale di Alta Matematica (INdAM).

\end{document}